\documentclass[11pt,leqno]{amsart}
\usepackage{amsmath,amssymb,amsfonts, amscd,   
pifont, amsthm,  amsxtra, latexsym, 
}
\usepackage{lmodern}
\usepackage[all]{xy}
\usepackage{here}
\usepackage{verbatim}
\usepackage{graphicx}

\usepackage{psfrag}
\usepackage{amsmath}
\usepackage{amsthm}
\usepackage{amsfonts}
\usepackage{amssymb}
\usepackage{extarrows} 
\usepackage{colortbl} 
\usepackage{pstricks,pst-arrow} 
\usepackage{fancybox}
\numberwithin{equation}{section} %

\usepackage[colorlinks=true,linkcolor=blue,citecolor=blue,urlcolor=blue]{hyperref}

\newtheorem{theorem}{Theorem}[section]

\newtheorem{proposition}[theorem]{Proposition}
\newtheorem{lemma}[theorem]{Lemma}
\newtheorem{corollary}[theorem]{Corollary}
\theoremstyle{definition}
\newtheorem{definition}[theorem]{Definition}
\newtheorem{example}[theorem]{Example}
\newtheorem{fact}[theorem]{Fact}

\theoremstyle{remark}
\newtheorem{remark}[theorem]{Remark}

\numberwithin{equation}{section} %

\usepackage[parfill]{parskip}

\begin{document}
\author{Toshiyuki KOBAYASHI}
\thanks{
Graduate School of Mathematical Sciences,
The University of Tokyo,
3-8-1 Komaba, Tokyo 153-8914, Japan,
and
French-Japanese Laboratory in Mathematics and its Interactions,
FJ-LMI CNRS IRL2025, Tokyo, Japan.
\email{toshi@ms.u-tokyo.ac.jp}
}

\title
[Stability of Branching Multiplicities]
{Stability of Branching Multiplicities
for Orthogonal Gelfand Pairs}

\begin{abstract}
We propose a structural framework for branching multiplicities in representation theory, emphasizing their behavior under variation of infinitesimal characters. 
For the orthogonal reductive pairs $(G,G')$ with complexified Lie algebras
$(\mathfrak{o}(n+1,\mathbb{C}), \mathfrak{o}(n,\mathbb{C}))$, we show that branching
multiplicities are governed by universal systems of linear inequalities on the
parameter space of reduced coherent families introduced in this paper.

To describe the loci where multiplicities may change, we introduce \emph{fences}: piecewise-linear hypersurfaces that divide the parameter space into convex regions. We prove that the multiplicity function is locally constant on each such region bounded by these fences.

The framework applies uniformly to finite-dimensional representations and to admissible smooth Fr\'echet representations of real reductive groups. It accounts for classical results such as the Weyl branching law and provides a unified explanation for a range of phenomena, including the Gross--Prasad conjecture, sporadic symmetry breaking operators, and fusion rules for Verma modules.

These results establish a general paradigm for branching multiplicities in orthogonal Gelfand pairs.
\end{abstract}
\maketitle

\setcounter{tocdepth}{1}
\tableofcontents
\section{Introduction}
Branching laws describe how representations of a group decompose upon restriction
to a subgroup.
A central problem in this theory is to understand the multiplicities
appearing in such decompositions.
While these multiplicities have been computed in a variety of individual cases,
a conceptual understanding of their global structure has long remained elusive.

To motivate our approach, we begin with four representative examples that exhibit a common underlying pattern:

\begin{enumerate}
\item[(1)] \textbf{(Weyl's branching law \cite{Weyl97})}  
Let $V_{\sigma}$ and $W_{\tau}$ denote irreducible finite-dimensional representations of $GL(n)$ and $GL(n-1)$, respectively, and let $[V_{\sigma}:W_{\tau}]$ denote the multiplicity of $W_{\tau}$ in $V_{\sigma}$. 
This multiplicity is either $0$ or $1$. 
Moreover, the nonvanishing condition is governed by the interlacing inequalities for highest weights $\sigma=(\sigma_1, \dots, \sigma_n)$ and $\tau=(\tau_1, \dots, \tau_{n-1})$:
\[
   [V_{\sigma}:W_{\tau}] \neq 0
\ \Longleftrightarrow \
\sigma_1 \ge \tau_1 \ge \sigma_2 \ge \cdots \ge \tau_{n-1} \ge \sigma_n.
\]
An analogous multiplicity-free branching law holds for the restriction $O(n) \downarrow O(n-1)$.

\item[(2)] \textbf{(Fusion rules for Verma modules)}  
Let $M(a)$, $M(b)$, and $M(c)$ be Verma modules of $\mathfrak{sl}(2,\mathbb{C})$. 
Generically, the multiplicity
\[
m(a,b,c)
= \dim \operatorname{Hom}_{\mathfrak{sl}(2,\mathbb{C})}\bigl(M(c), M(a)\otimes M(b)\bigr)
\]
is at most one. 
However, this multiplicity jumps to two precisely when the parameters satisfy the explicit inequalities
\[
 a + b + c \le -2, \qquad |a-b| \le -c -2,
\]
together with the parity conditions $a+b-c \in 2\mathbb{N}$ and $a,b,c \in \mathbb{Z}$.
This jumping phenomenon extends to the restriction of parabolic Verma modules for $\mathfrak{gl}(n,\mathbb{C})$ to $\mathfrak{gl}(n-1,\mathbb{C})$, where it occurs exactly at parameter values for which certain Jacobi polynomials vanish \cite[Thm.~8.1]{KP16}.

\item[(3)] \textbf{(The Gross--Prasad conjecture in the real case \cite{GP92})}
This conjecture generalizes Weyl's branching law and points toward a deep
understanding of the nonvanishing condition
$
[\Pi_{\lambda}:\pi_{\nu}] \neq 0,
$
where $\Pi_{\lambda}$ and $\pi_{\nu}$ are discrete series representations of
$G = O(p,q)$ and $G' = O(p-1,q)$, respectively.
For the pair $(G,G') = (U(p,q),\, U(p-1,q))$, a precise nonvanishing condition for
the multiplicity $[\Pi_\lambda:\pi_\nu]$ is described in \cite{He} by explicit
interleaving conditions on $\lambda$ and $\nu$, based on the combinatorics of the
theta correspondence.

\item[(4)] \textbf{(Sporadic conformally covariant operators)}  
On the space of differential forms on the pair of spheres $S^n \supset S^{n-1}$, 
there exist ``sporadic'' symmetry breaking operators that occur only at exceptional, countably many parameter values and cannot be realized as residues of rational families \cite[vol.~II]{xksbonvecI}. 
These operators form a new class, distinct from Juhl’s conformally covariant operators \cite{J09}. 
At such parameter values, the branching multiplicity for the pair $(O(n,1),O(n-1,1))$ jumps from zero to one.
\end{enumerate}

These examples point to a common structural phenomenon, yet a systematic understanding
of the underlying global structure has long been lacking.

\medskip

\textbf{Wall crossing versus fences.}
To understand how branching multiplicities depend on the parameters of representations,
one must note that, without crossing any classical walls associated with the variation
of representations of a \emph{single group}, the multiplicities in branching laws can
still vary when the parameters remain within a single dominant chamber.
The behavior of branching multiplicities is therefore more subtle, and depends essentially on the \emph{relative position of the infinitesimal characters of two groups}.
This distinction explains why traditional arguments based on Weyl chambers and wall-crossing fail to capture the stability regions and jumps observed in branching laws.

The boundaries across which branching multiplicities may change are therefore not walls in the conventional sense, but rather piecewise-linear hypersurfaces determined by linear relations between the infinitesimal characters of $G$ and $G'$.
These hypersurfaces admit a precise description in terms of the associated multi-signatures, and we refer to them as \emph{fences}.

\medskip

Rather than relying on representation-by-representation arguments, our aim is to identify a universal mechanism underlying such systems of inequalities and to formulate it in an intrinsic  framework. 
More precisely, we provide a uniform description of the loci in parameter space where branching multiplicities may change, expressed purely in terms of infinitesimal characters and independent of the choice of real form or of a particular realization of the representations.
To formulate our results for infinite-dimensional representations, the notion of multiplicity requires some care, as it depends intrinsically on the topology of the representations involved.

{\textbf{Multipllicity in the nonunitary setting.}}
Let $G$ be a real reductive Lie group, and let $\mathcal{M}(G)$ denote the category of smooth admissible representations of finite length and moderate growth, realized on Fr\'echet topological vector spaces \cite[Chap.~11]{Wa92}.  
By the Casselman--Wallach theory, this category is canonically equivalent to the category of $(\mathfrak{g},K)$-modules of finite length.  
We write $\operatorname{Irr}(G)$ for the set of irreducible objects in $\mathcal{M}(G)$.

For $\Pi \in \mathcal{M}(G)$ and $\pi \in \mathcal{M}(G')$, we define the space of \emph{symmetry breaking operators} by
\[
\operatorname{Hom}_{G'}(\Pi|_{G'}, \pi),
\]
namely, the space of continuous $G'$-homomorphisms with respect to the Fr\'echet topology.
The \emph{branching multiplicity} is defined as the dimension of this space, and is denoted by
\[
[\Pi|_{G'} : \pi], \quad \text{or simply } [\Pi : \pi].
\]

Our primary objective is to understand how these multiplicities vary as the representation parameters change.
As the four examples above illustrate, branching multiplicities are governed by a subtle and highly nontrivial interaction between the parameters of $G$ and those of $G'$, involving delicate parity conditions and intricate interleaving patterns, even when both vary within their respective dominant chambers.

\medskip
\noindent
\textbf{Reduced coherent families and parity.}
Motivated by this observation, we introduce the notion of a
\emph{reduced coherent family} (Definition~\ref{def:26032915}),
which is designed to capture, in an intrinsic manner, the parity conditions arising in branching laws.
Within this framework, parity conditions that appear in concrete branching problems are no longer imposed externally,
but are instead encoded intrinsically in the structure of the parameter space.
As a consequence, we establish a \emph{stability principle} for branching multiplicities,
showing that they remain constant on certain convex regions of the parameter space.

\medskip
\textbf{Multi-signature and interleaving patterns.}
To describe these convex regions, we introduce for each pair $(\lambda,\nu)$
of infinitesimal characters a \emph{multi-signature},
\[
\operatorname{sgn}(\lambda,\nu),
\]
which encodes the relative configuration of the infinitesimal characters of
$G$ and $G'$.

In terms of \emph{reduced coherent families} and
\emph{multi-signatures}, we formulate stability results for branching multiplicities without recourse to case-by-case parity assumptions,
and identify precisely the loci at which multiplicity jumps may occur, in terms of changes in the relative configuration of infinitesimal characters across fences.

\begin{theorem}[Stability of multiplicity; informal version.
See Theorem~\ref{thm:250515} for a precise formulation]
Let $G \supset G'$ be a real form of
$(O(n+1,\mathbb{C}), O(n,\mathbb{C}))$.
Let $\xi$ be a nonsingular weight, and let
$\{\Pi_{\lambda} : \lambda \in \Lambda(\xi)\}$
be a reduced coherent family of $G$-modules.
Let $\pi \in \mathcal{M}(G')$ have infinitesimal character $\nu$.
Assume that $(\xi,\nu)$ is away from fences.
Then the branching multiplicity
\[
   [\Pi_{\lambda}|_{G'} : \pi]
\]
is constant for all $\lambda \in \Lambda(\xi)$ satisfying
\begin{equation}
\label{eqn:multi_sgn}
   \operatorname{sgn}(\lambda,\nu)
   =
   \operatorname{sgn}(\xi,\nu).
\end{equation}
\end{theorem}

Here $\xi$ plays the role of a fixed base point determining a reduced coherent family,
while $\lambda$ varies within the associated parameter lattice.

For a given $\pi$, the branching multiplicity is often easy to determine at a suitably chosen
base point $\Pi_\xi$.
Theorem~1.1 then propagates this information:
it shows that the same multiplicity $[\Pi_\lambda|_{G'}:\pi]$ holds for every $\lambda$
sharing the same interleaving pattern with $\nu$.
In this way, Theorem~1.1 reduces the computation of branching multiplicities
to their determination at a single base point in each interleaving region.
\medskip

The equality \eqref{eqn:multi_sgn} of multi-signatures encodes an explicit system of linear inequalities in the parameter $\lambda$.  
These inequalities describe interlacing (or, more generally, \emph{interleaving}) patterns between the infinitesimal characters of $G$ and $G'$.  
Consequently, the parameter space decomposes into regions on which the branching multiplicity is constant; the boundaries of these convex regions are piecewise-linear hypersurfaces, which we refer to as \emph{fences}.

In contrast to classical wall-crossing phenomena, which are governed by inequalities involving a single infinitesimal character, fences arise from the \emph{relative} position of the infinitesimal characters of two groups.
These \emph{fences} yield a finer partition of the parameter space than that provided by dominant chambers alone, and may be viewed as a genuinely two-body analogue of wall-crossing phenomena, reflecting the intrinsically relative nature of branching problems.

The above stability principle recovers classical results such as Weyl’s branching law,
explains the occurrence of multiplicity jumps for Verma modules,
and provides a natural framework for Gross--Prasad multiplicities,
as well as for the appearance of sporadic symmetry breaking operators.
In this way, it unifies a range of phenomena that had previously been understood only in isolated or case-dependent settings.

Related results for pairs $(G,G')$ of real forms of
$(GL(n+1,\mathbb{C}), GL(n,\mathbb{C}))$,
including universal scalar identities and nonvanishing theorems for symmetry breaking operators,
were obtained in \cite{HKS, KS26} under stronger assumptions on the primary components arising in translation,
where the main emphasis was on applications to Shimura varieties and to branching laws for representations such as $A_{\mathfrak q}(\lambda)$.
By contrast, the present framework—based on reduced coherent families and multi-signatures—
is genuinely new and provides a refined structural description of branching multiplicities,
tailored to the orthogonal setting and extending beyond the scope of the general linear case.
The techniques introduced here have further consequences beyond the orthogonal setting,
including refinements of earlier results for the general linear case,
which will be discussed elsewhere.

The present paper is devoted to the conceptual and structural aspects of branching multiplicities.
Further applications of this framework, including new branching laws and additional sporadic phenomena,
will be explored in subsequent work.

\vskip 1pc
\par\noindent
{\bf{Convention.}}\enspace
We write
${\mathbb{N}}:=\{0,1,2,\dots,\}$ and
${\mathbb{N}}_+:=\{1,2,\dots,\}$.  

\section{Preliminaries}
\label{sec:pre}
In this section, we collect the basic notions and conventions used throughout the paper. 
Our focus is on infinitesimal characters, coherent continuation, and their reduced variants for real forms of $O(n+1,\mathbb{C})$. 
We pay particular attention to situations in which the group is not of Harish-Chandra class, and to formulating the relevant structures in a uniform fashion, independent of the choice of real form.

After a brief review of well-known material in Section~\ref{subsec:known}, we introduce \emph{reduced coherent families} of representations. 
These are tailored to tensoring with the standard representation and to working within a fixed dominant chamber, thereby avoiding wall-crossing phenomena. 
In order to describe regions of stability for branching multiplicities, we also develop the notions of \emph{fences} and \emph{interleaving patterns}, which encode the combinatorial structure governing changes in multiplicity.

These notions will be used in the next section to state and prove the main theorem, in which regions of stability for branching multiplicities are described explicitly in terms of fences.

\subsection{General Background}
\label{subsec:known}
In this subsection, we temporarily assume that $G$ is of Harish-Chandra
class.
This assumption is made solely in order to recall, in a concise form,
the well-established theory of coherent continuation due to
Duflo, Jantzen, Schmid, Vogan, and Zuckerman.
No new results are proved under this assumption.

In the subsequent subsections, we abandon the Harish-Chandra class
hypothesis and consider general real forms of $O(n+1,\mathbb{C})$.
There we introduce the notion of reduced coherent continuation, which 
is designed to capture the parity conditions arising in branching laws.

Throughout the paper, all universal enveloping algebras are understood to
be taken over the field $\mathbb{C}$.
Accordingly, we write $U(\mathfrak g)$ for the universal enveloping algebra
of the complexified Lie algebra $\mathfrak g_\mathbb C$, and denote its
center by $\mathfrak{Z}(\mathfrak g)$.

Let $G$ be a linear reductive Lie group with Lie algebra $\mathfrak{g}$.
Under this assumption, $G$ is of \emph{Harish-Chandra class}, that is,
\[
   \operatorname{Ad}(G) \subset \operatorname{Int}(\mathfrak{g}_{\mathbb{C}}).
\]
Let $\mathfrak{j}_{\mathbb{C}}$ be a Cartan subalgebra of $\mathfrak g_\mathbb C$,
and let $W_{\mathfrak g}$ denote the Weyl group of the root system
$\Delta(\mathfrak{g}_{\mathbb{C}}, \mathfrak{j}_{\mathbb{C}})$.
Then the Harish-Chandra isomorphism gives
\[
   \operatorname{Hom}_{\mathbb{C}\text{-alg}}
   (\mathfrak{Z}(\mathfrak g), \mathbb{C})
   \;\simeq\;
   \mathfrak j_\mathbb{C}^*/W_{\mathfrak{g}}, \quad \chi_\xi \leftrightarrow \xi .
\]

Let $V \in \mathcal M(G)$ be an admissible representation, and $\xi\in \mathfrak j_\mathbb C^*$.
The $\chi_\xi$-primary component of $V$ is defined as the space of generalized
$\chi_\xi$-eigenvectors:
\begin{equation}
\label{eqn:pr_gen_inf}
   P_{\xi}(V)
   := \bigcup_{k=1}^{\infty}
      \bigcap_{z \in \mathfrak{Z}(\mathfrak g)}
      \operatorname{Ker} (z - \chi_\xi(z))^{k}.
\end{equation}
Then $V$ admits a primary decomposition
\[
   V = \bigoplus_{\xi} P_{\xi}(V).
\]
By a mild abuse of notation, we denote by
$P_{\xi} \colon V \to P_{\xi}(V)$
the corresponding projection.

Let $\mathcal{V}(G)$ be the Grothendieck group of $\mathcal{M}(G)$; that is,
the abelian group generated by elements $X \in \mathcal{M}(G)$, subject to
the relation
$
   X = Y + Z
$
whenever there exists a short exact sequence
$
   0 \to Y \to X \to Z \to 0.
$

\begin{definition}[Coherent continuation of representations, Duflo \cite{D79}]
\label{def:coherent}
Let $G$ be a real reductive group of Harish-Chandra class.
Fix a maximally split Cartan subgroup $J \subset G$, and let
$\Lambda_f \subset \widehat{J}$ denote the lattice of weights of finite-dimensional representations of $G$.

Fix a weight $\xi \in \mathfrak{j}_{\mathbb{C}}^{*}$, and define 
\[
   \xi + \Lambda_f := \{\xi + \lambda : \lambda \in \Lambda_f\},
\]
where $\xi + \lambda$ is understood as a formal sum, since the sum is only defined at the level of differentials, i.e., $\xi + d \lambda$.

A \emph{coherent family} of virtual $G$-modules 
based on $\xi + \Lambda_f$ is a map
\[
   \Phi \colon \xi + \Lambda_f \to \mathcal{V}(G)
\]
satisfying the following conditions:
\begin{itemize}
\item
For each $\lambda \in \Lambda_f$, the virtual representation
$\Phi(\xi+\lambda)$ has infinitesimal character
$\xi + d\lambda \in \mathfrak{j}_{\mathbb{C}}^{*}$.

\item
For any finite-dimensional representation $F$ of $G$, one has
\[
   \Phi(\xi+\lambda) \otimes F
   =
   \sum_{\mu \in \Delta(F)} \Phi(\xi+\lambda+\mu)
   \quad \text{in } \mathcal{V}(G).
\]
\end{itemize}
\end{definition}

We recall some basic facts
(cf.\ \cite[Thm.~7.2.7, Cor.~7.2.27, Prop.~7.2.22]{V81}).

\begin{fact}
\label{fact:coherent}
\begin{enumerate}
\item
(Schmid, Zuckerman)
Let $\Pi \in \mathcal{V}(G)$ have infinitesimal character $\xi$.
Then there exists a coherent family based on $\xi + \Lambda_f$ such that
$\Phi(\xi) = \Pi$.

\item
If $\xi$ is nonsingular, then the coherent family in {\rm(1)} is unique.

\item
(Zuckerman \cite{Z77})
If $\Pi$ is an irreducible $G$-module with nonsingular infinitesimal character,
then $\Phi(\xi+\lambda)$ is nonzero and irreducible for all $\lambda$ such that
$\xi + d\lambda$ lies in the same strictly dominant chamber $D(\xi)_{>}$;
cf.\ \eqref{eqn:D+} below.
\end{enumerate}
\end{fact}

In Section~\ref{subsec:red_coh}, we introduce the notion of reduced coherent
continuation using only the standard representation of
$G \subset O(n+1,\mathbb{C})$, which need not be of Harish-Chandra class.
\medskip
\subsection{Some Conventions for $G \subset O(n+1, \mathbb C)$}
From now on, we assume that $G$ is a real form of
$O(n+1, \mathbb C)$, or a finite covering thereof, and set
\[
   r := \left\lfloor \frac{n+1}{2} \right\rfloor,
\]
the rank of $G$.
Typical examples include $G = O(p,q)$ or $SO(p,q)$ with $p+q = n+1$,
as well as $G = O^{\ast}(2m)$ or $SO^{\ast}(2m)$ with
$2m = n+1$.

Let $\mathfrak{g}_{\mathbb{C}}=\mathfrak{o}(n+1, \mathbb C)$ denote the
complexified Lie algebra.
Fix a Cartan subalgebra with standard dual basis
$e_1,\dots,e_r$.
We define
\[
\Delta_G:=
\{\pm e_i \pm e_j: 1\le i<j\le r\} \cup \{\pm e_i: 1\le i\le r\},
\]
which coincides with the root system of $\mathfrak{g}_{\mathbb{C}}$ when
$n$ is even.
When $n$ is odd, the set $\Delta_G$ is obtained from the root system of
$\mathfrak{g}_{\mathbb{C}}$ by adjoining the additional set
\[
\mathcal E := \{\pm e_i : 1\le i\le r\}.
\]
We adopt this convention throughout the paper, in particular for the notions of
nonsingularity and the dominant region $D(\xi)_{>}$.

Let $W_G$ denote the corresponding Weyl group:
\[
W_G :=
\mathfrak{S}_r \ltimes (\mathbb{Z}/2\mathbb{Z})^r.
\]
Let $\mathfrak{Z}(G)$ denote the subalgebra of the universal enveloping algebra
$U(\mathfrak{g})$ of $\mathfrak g_\mathbb C$ consisting of the
$O(n+1,\mathbb{C})$-invariant elements.
When $n$ is even, $\mathfrak{Z}(G)$ is a proper subalgebra of index two in
$\mathfrak{Z}(\mathfrak{g})$.
The projection onto the primary component introduced in the previous section
continues to be defined even when $G$ is not of Harish-Chandra class,
by replacing $\mathfrak{Z}(\mathfrak{g})$ with $\mathfrak{Z}(G)$.

\begin{remark}
\label{rem:ZG}
Our main results are formulated in terms of
$\mathfrak{Z}(G)$ rather than $\mathfrak{Z}(\mathfrak g)$.
This choice allows us to formulate the translation
principle for branching laws in a uniform manner, independently of the
real form, including cases where $G$ is not of Harish-Chandra class.
\end{remark}

The Harish-Chandra isomorphism yields a canonical bijection
\begin{equation}
\label{eqn:HCisom}
   \operatorname{Hom}_{\mathbb{C}\text{-alg}}
   \bigl(\mathfrak{Z}(G), \mathbb{C}\bigr)
   \;\simeq\;
   \mathbb{C}^{r}/W_G,
\end{equation}
normalized so that the infinitesimal character of the trivial one-dimensional
$\mathfrak{g}_{\mathbb{C}}$-module is represented by
\[
   \rho_{\mathfrak{o}(n+1)}
   =
   \Bigl(
      \tfrac{n-1}{2},
      \tfrac{n-3}{2},
      \dots,
      \tfrac{n+1}{2}-r
   \Bigr)
   \;\bmod\; W_G.
\]
For $\lambda = (\lambda_1, \dots, \lambda_r)$, we define the stabilizer subgroup
\[
   W_{\lambda}
   :=
   \{\, w \in W_G : w \lambda = \lambda \,\}
   \subset W_G.
\]

\begin{definition}
\label{def:lmd_reg}
A weight $\lambda$ is called \emph{nonsingular} if $W_{\lambda} = \{e\}$.
Equivalently, $\lambda$ is nonsingular if and only if
\begin{equation}
\label{eqn:lmd_reg}
   \lambda_a \neq 0
   \quad \text{for } 1 \le a \le r,
   \qquad
   \lambda_a \neq \pm \lambda_b
   \quad \text{for } a \neq b.
\end{equation}
\end{definition}

We define two norms on the space of weights by
\begin{equation}
\label{eqn:lmd_norm}
   |\lambda|
   :=
   \sum_{i=1}^{r} |\lambda_i|,
   \qquad
   \|\lambda\|
   :=
   \left(\sum_{i=1}^{r} \lambda_i^2\right)^{1/2}.
\end{equation}

\medskip
\subsection{Reduced Coherent Family of Representations}
\label{subsec:red_coh}

In this subsection, we introduce the notion of a
\emph{reduced coherent family of representations}.
Compared with the classical notion of a coherent family
(Definition~\ref{def:coherent}),
the reduced version is characterized by the following three features.

First, the group $G$ is not required to be of Harish-Chandra class.
Second, since no wall-crossing phenomena are involved, the parameter
space is analyzed in a refined manner within a single dominant chamber.
Third, the finite-dimensional representation $F$ used in the definition
is restricted to the natural (standard) representation.

The second condition suffices to formulate the translation principle for branching rules developed later in this paper, while the third condition captures the parity conditions appearing in branching rules.

Let $\xi$ be a nonsingular weight. 
The set of roots integral with respect to $\xi$ is defined by
\[
   R(\xi)
   :=
   \bigl\{
      \alpha \in \Delta_G :
      \langle \alpha^{\vee}, \xi \rangle \in \mathbb{Z}
   \bigr\}.
\]
We say that $\xi$ is \emph{integral} if
$
   R(\xi)
   =
   \Delta_G.
$

Associated with $\xi$, we introduce the following choice of positive roots and the corresponding open dominant chamber:
\begin{align}
\notag
   R^{+}(\xi)
   :=\,&
   \bigl\{
      \alpha \in \Delta_G
:      \langle \alpha^{\vee}, \xi \rangle \in \mathbb{N}_{+}
   \bigr\},
\\
\label{eqn:D+}
   D(\xi)_{>}
   :=\,&
   \bigl\{
      \eta \in \mathfrak{j}_{\mathbb{C}}^{*}
      :
      \langle \eta, \alpha \rangle > 0
      \text{ for all }
      \alpha \in R^{+}(\xi)
   \bigr\}.
\end{align}
\begin{remark}[Guiding convention on parameters]
Throughout this paper, we distinguish two roles played by infinitesimal characters.
A fixed nonsingular weight $\xi$ serves as a \emph{base point}, which determines a dominant chamber
and the associated reduced coherent family.
The symbol $\lambda$ denotes a \emph{varying parameter} in the affine lattice
$\xi+\mathbb{Z}^r$, indexing individual members of this family.
All translation phenomena and stability results are formulated by varying $\lambda$
while keeping the base point $\xi$ fixed.
\end{remark}

\begin{definition}[Reduced coherent family]
\label{def:26032915}
Let $\xi$ be nonsingular, and set
\begin{equation}
\label{eqn:260331_FJ}
   \Lambda(\xi)
   :=
   (\xi+\mathbb{Z}^r)\cap D(\xi)_{>}.
\end{equation}
A map
\[
   \Pi:\Lambda(\xi)\longrightarrow \mathcal{M}(G)
\]
is called a \emph{reduced coherent family} if the following conditions are satisfied:
\begin{itemize}
\item
for every $\lambda\in\Lambda(\xi)$, the object $\Pi_\lambda$ has infinitesimal
character $\lambda$;
\item
for $\lambda,\lambda+\varepsilon e_i\in\Lambda(\xi)$ with
$1\le i\le r$ and $\varepsilon\in\{+1,-1\}$, one has
\[
   P_{\lambda+\varepsilon e_i}\bigl(\Pi_\lambda\otimes F\bigr)
   \simeq
   \Pi_{\lambda+\varepsilon e_i}.
\]
\end{itemize}
\end{definition}

In the classical notion of coherent families, one considers tensor products with
arbitrary finite-dimensional representations of $G$, which leads to a fine
weight lattice. In the present setting, by contrast, we restrict attention to
tensoring with the standard representations $F$. As a result, the
relevant parameter lattice $\Lambda(\xi)$ is coarser than the full weight lattice.
This coarsening naturally incorporates the parity conditions that arise in the
formulation of stability for branching multiplicities. Moreover, this reduced
framework is flexible enough to apply to groups $G$ that lie outside the
Harish-Chandra class.
\medskip

Associated with the standard representation, we define the following
\emph{translation functors} on the category $\mathcal{M}(G)$:
\begin{equation}
\label{eqn:tr_functor}
   \psi_{\tau}^{\tau+\varepsilon e_i}(\,\cdot\,)
   :=
   P_{\tau+\varepsilon e_i}
   \bigl(P_{\tau}(\,\cdot\,)\otimes \mathbb{C}^{n+1}\bigr),
\end{equation}
where $\varepsilon\in\{+1,-1\}$ and $1\le i\le r$.

With this notation, the second condition in
Definition~\ref{def:26032915} can be equivalently expressed as
\begin{equation}
\label{eqn:rcf_transl}
   \psi_{\lambda}^{\lambda'}
   \bigl(\Pi_{\lambda}\bigr)
   \simeq
   \Pi_{\lambda'},
\end{equation}
for any $\lambda,\lambda'\in\Lambda(\xi)$ satisfying
$|\lambda-\lambda'|=1$.

When $G$ is of Harish-Chandra class, a reduced coherent family may be obtained by restricting the parameter space of an ordinary coherent
family.
\begin{proposition}[Case: $G$ of Harish-Chandra class]
\label{prop:260331}
Let $G \subset O(n+1,\mathbb{C})$ be a real form, or a finite covering
thereof, and suppose that $G$ is of Harish-Chandra class.

Let $J \subset G$ be a maximally split Cartan subgroup, and let $\Lambda$
be the lattice generated by the nonzero weights of the standard
representation.
Then there is a canonical identification
\begin{equation}
\label{eqn:wt_F}
   \mathbb{Z}^r \simeq \Lambda \subset \Lambda_f .
\end{equation}

Let $\Pi \in \operatorname{Irr}(G)$ be an irreducible representation with
nonsingular infinitesimal character $\xi$, and let
\[
   \Phi \colon \xi + \Lambda_f \longrightarrow \mathcal{V}(G)
\]
be the unique coherent family passing through $\Pi$ at $\xi$.
Then the restriction of $\Phi$ to
\[
   \Lambda(\xi)
   :=
   (\xi + \mathbb{Z}^r) \cap D(\xi)_{>}
\]
defines a reduced coherent family passing through $\Pi$.
\end{proposition}
\begin{proof}
Viewing both lattices as subsets of $\widehat{J}$, we may regard $\Lambda$
as a subgroup of $\Lambda_f$.
Moreover, the weights of $J$ occurring in the standard representation
$\mathbb{C}^{n+1}$ are uniquely determined by their differentials.
With respect to the standard basis $\{e_i\}_{1\le i\le r}$ of
$\mathfrak{j}_{\mathbb{C}}^{*}$, the set of nonzero weights of the standard
representation is given by
\[
   \mathcal{E}
   =
   \{\pm e_i : 1 \le i \le r\}.
\]
It follows that $\Lambda$ may be canonically identified with $\mathbb{Z}^r$.

By Fact~\ref{fact:coherent}(3), the restriction of the coherent family
$\Phi$ to $D(\xi)_{>}$ takes values in $\mathcal{M}(G)$ rather than in
$\mathcal{V}(G)$.
This completes the proof.
\end{proof}

The purpose of the following elementary example is to illustrate a phenomenon
specific to reduced coherent families, which already appears in the
finite-dimensional setting.
When $n$ is even, the set of finite-dimensional representations decomposes into
two distinct reduced coherent families, whereas in the odd-dimensional case,
the generic part (corresponding to $\lambda_r \neq 0$) forms a single reduced
coherent family.

This distinction reflects the fact that reduced coherent families incorporate,
in an intrinsic way, the parity conditions that arise in branching laws.
As will be made precise in the stability theorem below, this feature provides
the appropriate framework for describing translation-invariant multiplicity
phenomena.
In particular, in the Harish-Chandra class and for even $n$, it explains the
decomposition of $\widehat{G}$ into two components at the level of reduced
coherent families.

\begin{example}[Finite-dimensional case]
\label{ex:26032917}

\noindent
{\rm(1)}\enspace
Let $G = O(2r+1)$.
In this case, $G$ is of Harish-Chandra class, and
Proposition~\ref{prop:260331} applies.
Any irreducible representation $\Pi$ of $G$ remains irreducible upon
restriction to the subgroup $SO(2r+1)$, and the element $-I_{2r+1}$ acts by
a scalar, equal to either $1$ or $-1$.

We denote such a representation $\Pi$ by $F^{G}(\mu)_{\varepsilon}$ if the
highest weight of $\Pi|_{SO(2r+1)}$ is $\mu=(\mu_1,\dots,\mu_r)$, and
$-I_{2r+1}$ acts by the scalar
\[
   \varepsilon\,(-1)^{\sum_{i=1}^{r}\mu_i}
   \in \{1,-1\}.
\]

Let 
$
   \rho
   :=
   \left(r-\tfrac{1}{2}, \dots, \tfrac{3}{2}, \tfrac{1}{2}\right),
$
and define
\[
   \Lambda
   :=
   \left\{
      \lambda \in \left(\mathbb{Z} + \tfrac{1}{2}\right)^{r}
      : \lambda_1 > \lambda_2 > \cdots > \lambda_r > 0
   \right\}.
\]
Then, for each $\varepsilon \in \{1,-1\}$, the map
\[
   \Phi_{\varepsilon}
   \colon
   \Lambda \to \mathcal{M}(G),
   \qquad
   \lambda \mapsto F^{G}(\lambda - \rho)_{\varepsilon}
\]
defines a reduced coherent family, and one has
\[
   \Phi_{1}(\Lambda) \cup \Phi_{-1}(\Lambda) = \widehat{G}.
\]

\medskip
\noindent
{\rm(2)}\enspace
Let $G = O(2r)$.
For each $\mu = (\mu_1, \dots, \mu_r) \in \mathbb{Z}^{r}$ with
$\mu_1 \ge \cdots \ge \mu_r \ge 1$, there exists a unique irreducible
finite-dimensional representation $\Pi$ of $G$ such that the restriction
to $SO(2r)$ decomposes as
\[
   \Pi|_{SO(2r)}
   =
   F^{SO(2r)}(\mu) \oplus F^{SO(2r)}(\mu'),
\]
where
\[
   \mu' := (\mu_1, \dots, \mu_{r-1}, -\mu_r).
\]

We denote this representation $\Pi$ by $F^{G}(\mu)$.

Let
$
   \rho := (r-1, \dots, 1, 0),
$
and define
\[
   \Lambda
   :=
   \left\{
      \lambda \in \mathbb{Z}^{r}
      : \lambda_1 > \lambda_2 > \cdots > \lambda_r > 1
   \right\}.
\]
Then the map
\[
   \Phi \colon \Lambda \to \mathcal{M}(G),
   \qquad
   \lambda \mapsto F^{G}(\lambda - \rho)
\]
defines a reduced coherent family.
In this case, $\Phi(\Lambda)$ does not include representations with $\lambda_r = 0$,
equivalently $\mu_r=0$, and thus forms a proper subset of $\widehat{G}$.
\end{example}
\begin{remark} 
The nonvanishing of the multiplicity when $\mu_r = 0$ can be derived from the case where $\mu_r > 0$ by Proposition~\ref{prop:250623}.
\end{remark}
\subsection{Multi-Signatures, Fences and Interleaving Patterns}
\label{subsec:fence}

In this subsection, we introduce two key notions that will play a central role
in the statement of the main theorem in the next section: the
\emph{multi-signature}, which encodes the relative position of infinitesimal
characters for two groups, and the associated \emph{fences}, which describe the
loci where multiplicities may change.

Let $r$ and $s$ be positive integers, and let
$\lambda \in \mathbb{C}^{r}$ and $\nu \in \mathbb{C}^{s}$.
We denote by $\mathcal{P}(\lambda,\nu)$ the set of triples $(i,j,\delta)$
with $1 \le i \le r$, $1 \le j \le s$, and $\delta \in \{+,-\}$ such that
\[
   \lambda_i + \delta \nu_j \in \mathbb{Z} + \tfrac{1}{2}.
\]
Note that
\[
   \mathcal{P}(\lambda,\nu)
   =
   \mathcal{P}(\xi,\nu)
   \qquad
   \text{whenever } \lambda - \xi \in \mathbb{Z}^{r}.
\]

We define the associated multi-signature by
\begin{equation}
\label{eqn:lmd_nu_sgn}
   \operatorname{sgn}(\lambda,\nu)
   :=
   \bigl(
      \operatorname{sgn}
      (\lambda_i + \delta \nu_j)
   \bigr)_{(i,j,\delta) \in \mathcal{P}(\lambda,\nu)}.
\end{equation}

\begin{definition}[Fences and interleaving patterns]
\label{def:26040107}
Let $\xi \in \mathbb{C}^{r}$ be a nonsingular weight, and let
$\nu \in \mathbb{C}^{s}$.
A \emph{fence} is a hyperplane in the $(\lambda,\nu)$-space defined by
\[
   \xi_i + \delta \nu_j = \tfrac{1}{2}
   \quad\text{or}\quad
   -\tfrac{1}{2},
\]
for some $(i,j,\delta) \in \mathcal{P}(\xi,\nu)$.
We say that $(\xi,\nu)$ is \emph{away from fences} if
\[
   \xi_i + \delta \nu_j \notin \{\tfrac{1}{2}, -\tfrac{1}{2}\}
   \quad
   \text{for every } (i,j,\delta) \in \mathcal{P}(\xi,\nu).
\]
If the pair $(\xi,\nu)$ is away from all fences, we define the
\emph{interleaving region} by
\begin{equation}
\label{eqn:26040107}
   \mathcal{D}(\xi,\nu)
   :=
   \bigl\{
      \lambda \in \Lambda(\xi)
      :
      \operatorname{sgn}(\lambda,\nu)
      =
      \operatorname{sgn}(\xi,\nu)
   \bigr\}.
\end{equation}
\end{definition}

\medskip

Note that the region $\mathcal{D}(\xi,\nu) \subset \Lambda(\xi)$ can be described explicitly by
a system of inequalities expressing the relative ordering of the quantities
$\lambda_i$ and $\pm \nu_j$.
Such a description corresponds to an \emph{interlacing}, or more generally an
\emph{interleaving}, pattern between the coordinates of $\lambda$ and those of
$\nu$.

Under the assumptions of the definition, for each
$(i,j,\delta) \in \mathcal{P}(\xi,\nu)$ the quantity
$\xi_i + \delta \nu_j$ is a half-integer that is distinct from
$\pm\tfrac{1}{2}$.
Consequently, its sign is stable under shifts by $\pm\tfrac{1}{2}$, and one has
\[
   \operatorname{sgn}(\xi_i + \delta \nu_j)
   =
   \operatorname{sgn}(\xi_i + \delta \nu_j \pm \tfrac{1}{2}).
\]

\section{Main Results: $G \downarrow G'$}
\label{sec:main}

This section contains the main results of the paper.
We develop a new conceptual approach to the study of branching multiplicities, with particular emphasis on stability phenomena under translation of infinitesimal characters.

Our primary object of study is a pair of real reductive Lie groups
\[
   G \supset G',
\]
where $(G,G')$ 
is a real form of $(O(n+1, \mathbb C), O(n, \mathbb C))$, or a finite covering thereof.
Typical examples include the symmetric pairs
$(O(p,q), O(p-1,q))$, as well as the rank-one cases
$(SL(2,\mathbb{C}), SL(2,\mathbb{R}))$ and
$(SL(2,\mathbb{R}) \times SL(2,\mathbb{R}), SL(2,\mathbb{R}))$.

We establish a translation principle for the branching law
\[
   G \downarrow G'
\]
within a unified framework.
More precisely, working within the framework of reduced coherent
families, we show that branching multiplicities are locally constant on
convex regions of infinitesimal characters
determined by interleaving conditions.
This result is formulated in Theorem~\ref{thm:250515} purely in terms of
infinitesimal characters, and applies simultaneously to finite- and
infinite-dimensional representations, independently of the choice of real form of the pair $(G,G')$.

The stability phenomenon underlying this translation principle is a
consequence of Theorems~\ref{thm:25061911a} and~\ref{thm:25061911b}, which
constitute the main technical results of this paper.
These theorems provide explicit conditions on pairs of
$\mathfrak{Z}(G)$- and $\mathfrak{Z}(G')$-infinitesimal characters that
guarantee the nonvanishing of symmetry breaking operators after translation.

\subsection{Stability Theorems of Multiplicities}
\label{subsec:stability}

The following theorem asserts that branching multiplicities are locally
constant on regions determined by interleaving conditions.

\begin{theorem}[Stability of multiplicities]
\label{thm:250515}
Let $G \supset G'$ be a real form of
$(O(n+1,\mathbb{C}), O(n,\mathbb{C}))$, or of a finite covering thereof.

Let $\xi$ be a nonsingular weight, and let
$\pi \in \mathcal{M}(G')$ have
$\mathfrak{Z}(G')$-infinitesimal character $\nu$
 such that $(\xi,\nu)$ is away from fences.
Let $\{\Pi_{\lambda}\}_{\lambda \in \Lambda(\xi)} \subset \mathcal M(G)$ be a reduced coherent family
(Definition~\ref{def:26032915}).
Then the branching multiplicity
$
   [\Pi_{\lambda}|_{G'} : \pi]
$ is constant
for all $\lambda \in \mathcal{D}(\xi,\nu)$;
see Definition~\ref{def:26040107}.
\end{theorem}

In particular, branching multiplicities may vary only across the boundaries of
$\mathcal{D}(\xi,\nu)$, namely the fences.

\subsection{Universal Scalar Identities for Symmetry Breaking Operators}
\medskip
Set
\[
   r := \left\lfloor \frac{n+1}{2} \right\rfloor,
   \qquad
   s := \left\lfloor \frac{n}{2} \right\rfloor,
\]
which are the ranks of $G$ and $G'$, respectively.
Let $F = \mathbb{C}^{n+1}$ denote the standard representation of $G$, and
let $F' = \mathbb{C}$ denote the space of $G'$-invariant vectors.
We write
\begin{equation}
\label{eqn:pr_iota}    
   \operatorname{pr}_{F \to F'}
   \colon F \to \mathbb{C},
   \qquad
   \iota_{F' \to F}
   \colon \mathbb{C} \to F
\end{equation}
for the natural projection and embedding, respectively.

\begin{theorem}[Universal scalar identity]
\label{thm:25061911a}
There exist rational functions
\[
   C_{i,\varepsilon}(\lambda;\nu),
   \qquad
   1 \le i \le r,\ \varepsilon \in \{+,-\},
\]
depending only on the parameters
$\lambda=(\lambda_1,\dots,\lambda_r)$ and
$\nu=(\nu_1,\dots,\nu_s)$,
with the following property.

Let $\Pi\in\mathcal{M}(G)$ be a representation with nonsingular
$\mathfrak{Z}(G)$-infinitesimal character $\lambda$, and let
$\pi\in\mathcal{M}(G')$ be a representation with
$\mathfrak{Z}(G')$-infinitesimal character $\nu$.
Assume that
\begin{enumerate}
\item[$\bullet$]
$2\lambda_i+\varepsilon\neq 0$ when $n$ is even;
\item[$\bullet$]
The $\mathfrak{Z}(G)$-primary component
$P_{\lambda+\varepsilon e_i}(\Pi\otimes F)$ is a genuine eigenspace for the action of $\mathfrak{Z}(G)$.
\end{enumerate}
Let
$
   T:\Pi\longrightarrow\pi
$
be a symmetry breaking operator.
Then 
$C_{i,\varepsilon}(\lambda;\nu)$
is regular at $(\lambda,\nu)$ and one has
\[
   (T \otimes \operatorname{pr}_{F \to F'})
   \circ P_{\lambda+\varepsilon e_i}
   \circ (\operatorname{id}_\Pi \otimes \iota_{F' \to F})
   =
   C_{i,\varepsilon}(\lambda;\nu)\, T .
\]

Finally, the functions
$C_{i,\varepsilon}(\lambda;\nu)$
are \emph{universal} in the sense that they depend only on the
infinitesimal characters $\lambda$ and $\nu$, and are independent of the
choice of real forms $(G,G')$ as well as of the representations
$\Pi$ and $\pi$.
\end{theorem}
Next, we give an explicit determination of the rational functions
appearing in the universal scalar identity.

For $1 \le i \le r$, we define polynomials
$\varphi_{i,\varepsilon}(\lambda)$
in the variables $\lambda_1, \dots, \lambda_r$ by
\begin{equation}
\label{eqn:phi_ei_lmd}
\varphi_{i,\varepsilon}(\lambda)
:=
\begin{cases}
\displaystyle
2 \varepsilon\, \lambda_i
\prod_{\substack{1 \le j \le r \\ j \ne i}}
(\lambda_i - \lambda_j)(\lambda_i + \lambda_j),
& \text{if $n = 2r - 1$,}
\\[0.8em]
\displaystyle
(2\lambda_i + \varepsilon)\, \lambda_i
\prod_{\substack{1 \le j \le r \\ j \ne i}}
(\lambda_i - \lambda_j)(\lambda_i + \lambda_j),
& \text{if $n = 2r$.}
\end{cases}
\end{equation}

We also introduce polynomials
$g_{i,\varepsilon}(\lambda,\nu)$
in the variables
$\lambda_1,\dots,\lambda_r$
and
$\nu_1,\dots,\nu_s$,
for $\varepsilon\in\{+,-\}\equiv\{1,-1\}$ and $1\le i\le r$,
defined by
\begin{equation}
\label{eqn:gi_e}
g_{i,\varepsilon}(\lambda,\nu)
=
\begin{cases}
\displaystyle
\varepsilon\,\lambda_i
\prod_{j=1}^{r-1}
(\lambda_i-\nu_j+\tfrac{1}{2}\varepsilon)
(\lambda_i+\nu_j+\tfrac{1}{2}\varepsilon),
& \text{if $n = 2r - 1$,}
\\[1.2em]
\displaystyle
\prod_{j=1}^{r}
(\lambda_i-\nu_j+\tfrac{1}{2}\varepsilon)
(\lambda_i+\nu_j+\tfrac{1}{2}\varepsilon),
& \text{if $n = 2r$.}
\end{cases}
\end{equation}

\begin{theorem}[Explicit formula for the universal constants]
\label{thm:25061911b}
The rational functions
$C_{i,\varepsilon}(\lambda;\nu)$
appearing in Theorem~\ref{thm:25061911a}
are given explicitly by
\begin{equation}
\label{eqn:26032009}
   C_{i,\varepsilon}(\lambda;\nu)
   =
   \frac{g_{i,\varepsilon}(\lambda,\nu)}
   {\varphi_{i,\varepsilon}(\lambda)}.
\end{equation}
Here $\varphi_{i,\varepsilon}(\lambda)$ denotes the polynomial defined in
\eqref{eqn:phi_ei_lmd}, which is nonvanishing under the assumptions of
Theorem~\ref{thm:25061911a}.
\end{theorem}
As a consequence of Theorems~\ref{thm:25061911a}
and~\ref{thm:25061911b}, we obtain the following result.

\begin{theorem}
\label{thm:25051110}
Let $\Pi \in \mathcal{M}(G)$ have a nonsingular
$\mathfrak{Z}(G)$-infinitesimal character
$\lambda$, and let
$\pi \in \mathcal{M}(G')$ have
$\mathfrak{Z}(G')$-infinitesimal character
$\nu$.
Let
$
   T \colon \Pi \longrightarrow \pi
$
be a symmetry breaking operator.

Fix $\varepsilon \in \{+,-\}$ and an index $1 \le i \le r$.
Assume that the $\mathfrak{Z}(G)$-primary component
$P_{\lambda+\varepsilon e_i}(\Pi \otimes F)$
is a genuine eigenspace for the action of $\mathfrak{Z}(G)$.
If
\begin{equation}
\label{eqn:25061912}
   \lambda_i + \tfrac{1}{2}\varepsilon
   \notin
   \{\pm \nu_j : 1 \le j \le s\}
   \quad
   (\cup \{0\} \text{ when $n$ is even}),
\end{equation}
then the operator
\[
   T \otimes \operatorname{pr}_{F \to F'}
\]
does not vanish on the primary component
\[
   P_{\lambda+\varepsilon e_i}(\Pi \otimes F)
   =
   \psi_{\lambda}^{\lambda+\varepsilon e_i}(\Pi).
\]
\end{theorem}
\begin{proof}
[Proof of Theorem~\ref{thm:25051110}
assuming Theorems~\ref{thm:25061911a}
and~\ref{thm:25061911b}]
The assumption~\eqref{eqn:25061912} implies that
$g_{i,\varepsilon}(\lambda,\nu) \neq 0$.
The conclusion therefore follows immediately from
Theorems~\ref{thm:25061911a} and~\ref{thm:25061911b}.
\end{proof}
\begin{remark}
\label{rem:250619}
As will be clear from the proof given later, the second assumption in
Theorem~\ref{thm:25061911a}, as well as the assumption in
Theorem~\ref{thm:25051110}, can be relaxed to the requirement that the Casimir
element $c_G$ act by scalar multiplication on the
$\mathfrak{Z}(G)$-primary component
$P_{\lambda+\varepsilon e_i}(\Pi\otimes F)$.
\end{remark}
\subsection{Proof of the Stability Theorem}

In this subsection, assuming Theorems~\ref{thm:25061911a}
and~\ref{thm:25061911b}, we complete the proof of the stability theorem
for branching multiplicities.
More precisely, since the first implication has already been established,
the argument reduces to the following chain of implications:
\[
   \text{Theorems~\ref{thm:25061911a} and~\ref{thm:25061911b}}
   \;\Longrightarrow\;
   \text{Theorem~\ref{thm:25051110}}
   \;\Longrightarrow\;
   \text{Theorem~\ref{thm:250515}}.
\]

Thus, the proof of Theorem~\ref{thm:250515} ultimately rests on the
universal scalar identities established in
Theorems~\ref{thm:25061911a} and~\ref{thm:25061911b}.
After this subsection, the remainder of the paper is devoted to the
proofs of these two theorems.

Accordingly, we focus here on the final step, namely the implication
\[
   \text{Theorem~\ref{thm:25051110}}
   \;\Longrightarrow\;
   \text{Theorem~\ref{thm:250515}}.
\]
To this end, we first prepare a pair of auxiliary lemmas.

Recall that $\{e_i\}_{1 \le i \le r}$ denotes the standard basis of
$\mathbb{Z}^r$, and 
\[
   \mathcal{E}
   =
   \{\pm e_i : 1 \le i \le r\}
   \subset \mathbb{Z}^r.
\]

\begin{proposition}
\label{prop:250623}
Let $\Pi \in \mathcal{M}(G)$ have a nonsingular
$\mathfrak{Z}(G)$-infinitesimal character $\lambda$, and let
$\pi \in \mathcal{M}(G')$ have
$\mathfrak{Z}(G')$-infinitesimal character $\nu$.
Let $\lambda' \in \mathbb{C}^r$ satisfy
$\lambda' - \lambda \in \mathcal{E}$.

\noindent
{\rm(1)}\enspace
Assume that $\lambda$ satisfies \eqref{eqn:25061912} and that the
$\mathfrak{Z}(G)$-primary component
\[
   P_{\lambda'}(\Pi \otimes F)
\]
is a genuine eigenspace.
Then one has
\[
   \bigl[
      \psi_{\lambda}^{\lambda'}(\Pi)\big|_{G'} : \pi
   \bigr]
   \ge
   \bigl[
      \Pi\big|_{G'} : \pi
   \bigr].
\]

\medskip
\noindent
{\rm(2)}\enspace
Assume further that
\[
   P_{\lambda}
   \bigl(
      P_{\lambda'}(\Pi \otimes F) \otimes F
   \bigr)
\]
is a genuine eigenspace of $\mathfrak{Z}(G)$.
Then
\[
   \bigl[
      \psi_{\lambda'}^{\lambda}
      \circ
      \psi_{\lambda}^{\lambda'}(\Pi)\big|_{G'} : \pi
   \bigr]
   \ge
   \bigl[
      \psi_{\lambda}^{\lambda'}(\Pi)\big|_{G'} : \pi
   \bigr].
\]
\end{proposition}

\begin{proof}
Since $\lambda' = \lambda + \varepsilon e_i$ for some
$1 \le i \le r$ and $\varepsilon \in \{+,-\}$, we may apply
Theorem~\ref{thm:25051110}.
It follows that the map
\[
   \operatorname{Hom}_{G'}
   \bigl(\Pi|_{G'}, \pi\bigr)
   \longrightarrow
   \operatorname{Hom}_{G'}
   \bigl(
      P_{\lambda'}(\Pi \otimes F)\big|_{G'}, \pi
   \bigr),
\]
given by
\[
   T \longmapsto
   (T \otimes \operatorname{pr}_{F \to F'})
   \circ
   P_{\lambda'},
\]
is injective.

By definition \eqref{eqn:tr_functor}, we have
$   \psi_{\lambda}^{\lambda'}(\Pi)
   \simeq
   P_{\lambda'}(\Pi \otimes F).
$
This proves the first inequality.

For the second inequality, we interchange the roles of
$\Pi$, $\lambda$, and $\varepsilon$ with those of
$P_{\lambda'}(\Pi)$, $\lambda'$, and $-\varepsilon$, respectively.
Then
\[
   {\lambda'}_i + \tfrac{1}{2}(-\varepsilon)
   =
   \lambda_i + \tfrac{1}{2}\varepsilon,
\]
so that the assumption \eqref{eqn:25061912} remains valid.
The second inequality therefore follows by the same argument.
\end{proof}

The following lemma relies on the convexity of $\mathcal{D}(\xi,\nu)$.
In particular, any two elements in the interleaving region
$\mathcal{D}(\xi,\nu)$
can be connected by a path consisting of unit steps in the lattice
$\xi + \mathbb{Z}^{r}$.
\begin{lemma}
\label{lem:26040112}
Let $\xi$ be a nonsingular weight, and suppose that both
$(\xi,\nu)$ and $(\lambda,\nu)$ belong to $\mathcal{D}(\xi,\nu)$.
Then there exists a sequence
\[
   \lambda^{(j)} \in \xi + \mathbb{Z}^{r},
   \qquad j = 0,1,\dots,N,
\]
such that
\[
   \lambda^{(0)} = \xi,
   \qquad
   \lambda^{(N)} = \lambda,
   \qquad
   \lambda^{(j)} - \lambda^{(j-1)} \in \mathcal{E},
\]
and
\[
   (\lambda^{(j)},\nu) \in \mathcal{D}(\xi,\nu)
   \quad
   \text{for all } 1 \le j \le N.
\]
\end{lemma}

\begin{proof}
The existence of a sequence
$\{\lambda^{(j)}\}_{0 \le j \le N}$ as in the statement defines an
equivalence relation on the set of nonsingular dominant elements in
$
   \Lambda(\xi)
   =
   (\xi + \mathbb{Z}^{r}) \cap D(\xi)_{>}
$;
see \eqref{eqn:260331_FJ}.
We prove the lemma by induction on the norm $|\lambda-\xi|$,
defined in \eqref{eqn:lmd_norm}.

Assume that $|\lambda-\xi| \ge 1$.
Then there exist an index $1 \le i \le r$ and
$\varepsilon \in \{+,-\}$ such that
\[
   |(\lambda + \varepsilon e_i) - \xi|
   <
   |\lambda - \xi|.
\]

The region $\mathcal{D}(\xi,\nu)$ is defined by linear inequalities in $\lambda$,
including both the interleaving condition
$\operatorname{sgn}(\lambda,\nu)=\operatorname{sgn}(\xi,\nu)$
(Definition~\ref{def:26040107})
and the dominance condition defining $D(\xi)_{>}$.
Hence these inequalities are preserved under the move
$\lambda \mapsto \lambda + \varepsilon e_i$
whenever they hold for both $\lambda$ and $\xi$.
Indeed,
\begin{alignat*}{3}
   \lambda_i \ge {}& (\lambda+\varepsilon e_i)_i
    \ge {}& \xi_i,
   \qquad &\text{if } \varepsilon = -, \\
   \lambda_i \le {}& (\lambda+\varepsilon e_i)_i
    \le {}& \xi_i,
   \qquad &\text{if } \varepsilon = +.
\end{alignat*}
Therefore, $\lambda + \varepsilon e_i \in \mathcal{D}(\xi,\nu)$.

The assertion now follows by induction.
\end{proof}

\begin{proof}
[Proof of Theorem~\ref{thm:250515} assuming Theorem~\ref{thm:25051110}]
By Lemma~\ref{lem:26040112}, it suffices to prove Theorem~\ref{thm:250515}
in the case where $\lambda - \xi \in \mathcal{E}$.
In this situation, we have
\[
   \psi_{\xi}^{\lambda}(\Pi_{\xi})
   \simeq
   \Pi_{\lambda},
   \qquad
   \psi_{\lambda}^{\xi}(\Pi_{\lambda})
   \simeq
   \Pi_{\xi},
\]
since
$
   \Pi \colon D(\xi)_{>} \longrightarrow \mathcal{M}(G)
$
is a reduced coherent family.
Therefore, all the assumptions of Proposition~\ref{prop:250623} are satisfied, and hence
\[
   [\Pi_{\lambda}|_{G'} : \pi]
   =
   [\Pi_{\xi}|_{G'} : \pi].
\]
This completes the proof of Theorem~\ref{thm:250515}.
\end{proof}
The remaining part of the paper is to 
give a proof of Theorems~\ref{thm:25061911a} and~\ref{thm:25061911b} based on
Theorem~\ref{thm:25061808}.

\section{Projection by Polynomials of the Casimir Operator}
\label{sec:Casimir}

This section constitutes the first step in the proof of
Theorems~\ref{thm:25061911a} and~\ref{thm:25061911b}.
Although translation functors for higher-rank groups $G$ are in general
highly nontrivial, the key observation in the present setting is that
their effect can be controlled solely through the action of the Casimir element.
The goal of this section is to establish this reduction, which allows us
to bypass the full complexity of translation theory and to deduce the relevant properties
from the spectral behavior of the Casimir element.

\subsection{Projection by a Polynomial of the Casimir Element}

Let $\mathfrak{g} = \mathfrak{o}(n+1)$, and define elements
$X_{ij} \in \mathfrak{g}$ by
\begin{equation}
\label{eqn:Xij}
   X_{ij} := E_{ij} - E_{ji},
   \qquad
   0 \le i \neq j \le n .
\end{equation}
We fix the normalization of the Casimir element $c_G$ by
\begin{equation}
\label{eqn:cG}
   c_G := - \sum_{0 \le i < j \le n} X_{ij}^2 .
\end{equation}
On the standard representation $F = \mathbb{C}^{n+1}$,
the Casimir element $c_G$ acts by the scalar $n$.
More generally, if $\Pi \in \mathcal{M}(G)$ has
$\mathfrak{Z}(G)$-infinitesimal character $\lambda$,
then $c_G$ acts on $\Pi$
by the scalar $\|\lambda\|^2 - \|\rho_{\mathfrak g}\|^2$.

The set of weights of $F$ in $\mathfrak{j}_\mathbb C^*$ is given by
\[
   \Delta(F)
   =
   \begin{cases}
      \mathcal{E},
      & \text{if $n$ is odd}, \\
      \mathcal{E} \cup \{0\},
      & \text{if $n$ is even},
   \end{cases}
\]
where $\mathcal{E} = \{\pm e_i : 1 \le i \le r\}$.

For $1 \le i \le r$ and $\varepsilon \in \{+,-\}$,
 we introduce a
$\mathfrak{Z}(G)$-valued polynomial in $\lambda$ by
\begin{equation}
\label{eqn:phi_ae}
   \phi_{\lambda}^{\lambda + \varepsilon e_i}
   :=
   \prod_{\mu \in \Delta(F)\setminus\{\varepsilon e_i\}}
   \frac{1}{2}
   \bigl(
      c_G - \|\lambda+\mu\|^2 + \|\rho_{\mathfrak g}\|^2
   \bigr)
   \in \mathfrak{Z}(G)[\lambda].
\end{equation}

\subsection{Separation of Eigenvalues
by the Casimir Operator}
\par
We recall that $r=\left\lfloor \frac{n+1}{2} \right\rfloor$ is the rank of $G$.
Although the center $\mathfrak{Z}(G)$ is generated by $r$
algebraically independent elements,
a single central element—namely, the Casimir element $c_G$—is sufficient
to separate distinct infinitesimal characters
in the tensor product representation $\Pi \otimes F$
 when $F$ is the standard representation.

This observation leads to an explicit identity governing symmetry
breaking under translation, which is formulated in
Theorem~\ref{thm:25061808} below.
Our main results—Theorems~\ref{thm:25061911a}
and~\ref{thm:25061911b}—are derived from
Theorem~\ref{thm:25061808}; see
Section~\ref{subsec:pf_thm}.

Recall from
\eqref{eqn:gi_e} and
\eqref{eqn:phi_ei_lmd}
the definitions of the polynomials
$g_{i,\varepsilon}(\lambda,\nu)$ and
$\varphi_{i,\varepsilon}(\lambda)$, respectively.

In the following statement,
we do not assume that the Casimir element $c_G$ acts on
$P_{\lambda+\varepsilon e_i}(\Pi \otimes F)$
by scalar multiplication.
\begin{theorem}
\label{thm:25061808}
Let $\Pi \in \mathcal{M}(G)$ have
$\mathfrak{Z}(G)$-infinitesimal character
$\lambda$,
let $\pi \in \mathcal{M}(G')$ have
$\mathfrak{Z}(G')$-infinitesimal character
$\nu$,
and let
$
   T \colon \Pi \to \pi
$
be a symmetry breaking operator.
Then, for any $\varepsilon \in \{+,-\}$,
any $1 \le i \le r$,
and any $N \in \mathbb{N}_{+}$,
one has
\begin{equation}
\label{eqn:Tgi_N}
   (T \otimes \operatorname{pr}_{F \to F'})
   \circ
   (\phi_{\lambda}^{\lambda+\varepsilon e_i})^{N}
   \circ
   (\operatorname{id}_\Pi \otimes \iota_{F' \to F})
   =
   \varphi_{i,\varepsilon}(\lambda)^{N-1}
   g_{i,\varepsilon}(\lambda,\nu)\,
   T.
\end{equation}
\end{theorem}

In particular, in the case $N=1$, the formula simplifies to:
\begin{equation}
\label{eqn:Tgi}
  (T \otimes \operatorname{pr}_{F \to F'})\circ
   \phi_{\lambda}^{\lambda + \varepsilon e_i}
   \circ (\operatorname{id}_\Pi \otimes \iota_{F' \to F})
  =g_{i, \varepsilon}(\lambda, \nu)  T.  
\end{equation}

\subsection{Translation Functors Using Only the Casimir Operator}

For $1 \le i \le r$, we define a homogeneous polynomial of degree $2r-1$ by
\begin{equation}
\label{eqn:hi_lmd}
   h_i(\lambda)
   :=
   \lambda_i
   \prod_{\substack{1 \le j \le r \\ j \ne i}}
   (\lambda_i - \lambda_j)(\lambda_i + \lambda_j).
\end{equation}

By definition, $h_i(\lambda) \neq 0$ if $\lambda$ is nonsingular
(see \eqref{eqn:lmd_reg}), and the polynomials
$\varphi_{i,\varepsilon}(\lambda)$
in the variables $\lambda_1, \dots, \lambda_r$,
defined in \eqref{eqn:phi_ei_lmd}, are expressed as
\begin{equation*}
   \varphi_{i,\varepsilon}(\lambda)
   =
   \begin{cases}
      2 \varepsilon\, h_i(\lambda),
      & \text{if $n = 2r - 1$,}
      \\
      (2\lambda_i + \varepsilon)\, h_i(\lambda),
      & \text{if $n = 2r$.}
   \end{cases}
\end{equation*}

We recall from \eqref{eqn:pr_gen_inf} that $P_{\tau}$ denotes the
projection operator onto the $\tau$-primary component.
Under a mild assumption, the action of
$\phi_{\lambda}^{\lambda+\varepsilon e_i}
\in \mathfrak{Z}(G)[\lambda]$
on $\Pi \otimes F$ coincides, up to a scalar factor, with the projection operator $P_{\lambda+\varepsilon e_i}$.
\begin{proposition}
\label{prop:25042720}
Suppose that $\Pi \in \mathcal{M}(G)$ has
$\mathfrak{Z}(G)$-infinitesimal character $\lambda$,
and assume that all generalized eigenspaces of
$\mathfrak{Z}(G)$ in $\Pi \otimes F$
are genuine eigenspaces.
Let $\varepsilon \in \{+,-\}$ and $1 \le i \le r$.
Then the following identity holds in
$\operatorname{End}(\Pi \otimes F)$:
\[
   \phi_{\lambda}^{\lambda + \varepsilon e_i}
   =
   \varphi_{i,\varepsilon}(\lambda)\,
   P_{\lambda + \varepsilon e_i}.
\]
\end{proposition}

\begin{proof}
The tensor product representation $\Pi \otimes F$
has finite length and therefore admits a primary decomposition:
\begin{equation}
\label{eqn:piF_primary}
   \Pi \otimes F
   =
   \bigoplus_{\tau \in \Delta(F)}
   P_{\lambda+\tau}(\Pi \otimes F)
   \quad
   \text{in $\mathcal{M}(G)$.}
\end{equation}
By assumption, each generalized eigenspace
$P_{\lambda+\tau}(\Pi \otimes F)$
is in fact an eigenspace of the Casimir element $c_G$,
with eigenvalue
$\|\lambda+\tau\|^2 - \|\rho_{\mathfrak g}\|^2$
for all $\tau \in \Delta(F)$.

Hence, the Casimir element $c_G$ acts on
$P_{\lambda+\tau}(\Pi \otimes F)$
by scalar multiplication with
$\|\lambda + \tau\|^2 - \|\rho_\mathfrak{g}\|^2$.
It then follows directly from the definition \eqref{eqn:phi_ae}
that $\phi_{\lambda}^{\lambda+\varepsilon e_i}$
vanishes on all $(\lambda+\tau)$-primary components
except when $\tau = \varepsilon e_i$.

To compute the action of
$\phi_{\lambda}^{\lambda+\varepsilon e_i}$
on the remaining component
$P_{\lambda+\varepsilon e_i}(\Pi \otimes F)$,
set
\[
   \Xi := \{-1,1\} \times \{1,2,\dots,r\}.
\]
Suppose that $\tau \in \Delta(F) \setminus \{\varepsilon e_i\}$.
Then either $\tau = b e_j$
for some $(b,j) \in \Xi \setminus \{(\varepsilon,i)\}$,
or $\tau = 0$ when $n$ is even.
It follows that
$\phi_{\lambda}^{\lambda+\varepsilon e_i}$
acts on
$P_{\lambda+\varepsilon e_i}(\Pi \otimes F)$
by the scalar
\[
\begin{cases}
   \displaystyle
   \prod_{(b,j) \in \Xi \setminus \{(\varepsilon,i)\}}
   (\varepsilon \lambda_i - b\lambda_j),
   & \text{if $n$ is odd}, \\[6pt]
   \displaystyle
   (\varepsilon \lambda_i + \tfrac{1}{2})
   \prod_{(b,j) \in \Xi \setminus \{(\varepsilon,i)\}}
   (\varepsilon \lambda_i - b\lambda_j),
   & \text{if $n$ is even}.
\end{cases}
\]
This scalar agrees with $\varphi_{i,\varepsilon}(\lambda)$
by definition \eqref{eqn:phi_ei_lmd}.
This completes the proof of
Proposition~\ref{prop:25042720}.
\end{proof}
The assumption in Proposition~\ref{prop:25042720}
that all generalized eigenspaces of $\mathfrak{Z}(G)$
in $\Pi \otimes F$ are genuine eigenspaces
can be relaxed as follows.
\begin{proposition}
\label{prop:25061805}
{\rm(1)}\enspace
Suppose that $\Pi \in \mathcal{M}(G)$ has
$\mathfrak{Z}(G)$-infinitesimal character $\lambda$.
Then, for any sufficiently large integer $N$,
the operator $(\phi_{\lambda}^{\lambda+\varepsilon e_i})^{N}$
maps $\Pi \otimes F$ into the primary component
$P_{\lambda+\varepsilon e_i}(\Pi \otimes F)$,
for every $\varepsilon \in \{+,-\}$ and
$1 \le i \le r$.

\medskip
\noindent
{\rm(2)}\enspace
Let $\varepsilon \in \{+,-\}$ and $1 \le i \le r$,
and suppose in addition that the Casimir element $c_G$
acts by scalar multiplication on the primary component
$P_{\lambda+\varepsilon e_i}(\Pi \otimes F)$.
Then, for any sufficiently large integer $N$,
the following identity holds in
$\operatorname{End}(\Pi \otimes F)$:
\[
   (\phi_{\lambda}^{\lambda+\varepsilon e_i})^{N}
   =
   \varphi_{i,\varepsilon}(\lambda)^{N}\,
   P_{\lambda+\varepsilon e_i}.
\]
\end{proposition}

\begin{remark}
If $\Pi \in \operatorname{Irr}(G)$,
one may take $N = n+1$ in
Proposition~\ref{prop:25061805}.
\end{remark}

A central feature of Proposition~\ref{prop:25061805}
is that, despite the higher rank of the group $G$,
the Casimir operator alone suffices to separate the primary components
of $\Pi \otimes F$.

\begin{proof}[Proof of Proposition~\ref{prop:25061805}]
Since $\Pi \otimes F$ is a $G$-module of finite length,
there exists $\ell \in \mathbb{N}_{+}$
such that, for every $\tau \in \Delta(F)$,
the primary component (see \eqref{eqn:pr_gen_inf})
associated with the
$\mathfrak{Z}(G)$-infinitesimal character $\lambda+\tau$
is given by
\[
   P_{\lambda+\tau}(\Pi \otimes F)
   =
   \bigcup_{k=1}^{\ell}
   \bigcap_{z \in \mathfrak{Z}(G)}
   \operatorname{Ker}\bigl(z-\chi_{\lambda+\tau}(z)\bigr)^k.
\]

It follows that, for any integer $N \ge \ell$,
the operator
\[
   (\phi_{\lambda}^{\lambda+\varepsilon e_i})^N
\]
annihilates each summand
$P_{\lambda+\tau}(\Pi \otimes F)$
whenever $\tau = b e_j$
with $(b,j) \in \Xi \setminus \{(\varepsilon,i)\}$,
or $\tau = 0$ when $n$ is even,
as already shown in the proof of
Proposition~\ref{prop:25042720}.

On the other hand, by assumption,
the Casimir element $c_G$ acts by scalar multiplication
on the remaining primary component
\[
   P_{\lambda+\varepsilon e_i}(\Pi \otimes F)
\]
in the primary decomposition \eqref{eqn:piF_primary}.
Therefore, the operator
$(\phi_{\lambda}^{\lambda+\varepsilon e_i})^N$
acts on
$P_{\lambda+\varepsilon e_i}(\Pi \otimes F)$
by scalar multiplication by
$\varphi_{i,\varepsilon}(\lambda)^N$,
as computed in the proof of
Proposition~\ref{prop:25042720}.
\end{proof}

\begin{corollary}
\label{cor:250624}
Suppose that $\Pi \in \mathcal{M}(G)$ has a nonsingular
$\mathfrak{Z}(G)$-infinitesimal character $\lambda$
(see \eqref{eqn:lmd_reg}).
Let $\varepsilon \in \{+,-\}$ and $1 \le i \le r$,
and suppose in addition that the Casimir element $c_G$
acts by scalar multiplication on the primary component
$P_{\lambda+\varepsilon e_i}(\Pi \otimes F)$.
Assume moreover that
$2\lambda_i+\varepsilon \ne 0$ when $n$ is even.
Then the projection operator
$P_{\lambda+\varepsilon e_i}$
can be realized as a scalar multiple of
$(\phi_{\lambda}^{\lambda+\varepsilon e_i})^N$
for sufficiently large $N$.

In particular,
\[
   (\phi_{\lambda}^{\lambda+\varepsilon e_i})^N
   (\Pi \otimes \mathbb{C}^{n+1})
   \simeq
   \psi_{\lambda}^{\lambda+\varepsilon e_i}(\Pi),
\]
where $\psi_{\lambda}^{\lambda+\varepsilon e_i}$
denotes the translation functor defined in
\eqref{eqn:tr_functor}.
\end{corollary}

\begin{proof}
Under the assumptions, we have
$\varphi_{i,\varepsilon}(\lambda) \neq 0$.
The assertion follows immediately from
Proposition~\ref{prop:25061805}(2).
\end{proof}

\subsection{Reduction Step: Proof of the Theorems}
\label{subsec:pf_thm}
The proof of Theorem~\ref{thm:25061808} will be given in
Section~\ref{sec:pf_main}.
In the present subsection, we establish the first implication in the
following chain; the remaining implications have already been proved
in the preceding sections:
\[
  \text{Theorem~\ref{thm:25061808}}
  \;\Longrightarrow\;
  \text{Theorems~\ref{thm:25061911a} and~\ref{thm:25061911b}}
  \;\Longrightarrow\;
  \text{Theorem~\ref{thm:25051110}}
  \;\Longrightarrow\;
  \text{Theorem~\ref{thm:250515}} .
\]
\begin{proof}
[Proof of Theorems~\ref{thm:25061911a} and~\ref{thm:25061911b}
assuming Theorem~\ref{thm:25061808}]
The regularity assumption on $\lambda$ guarantees that the polynomial
$h_i(\lambda)$ defined in \eqref{eqn:hi_lmd} does not vanish.
Hence, the nonvanishing of
$\varphi_{i,\varepsilon}(\lambda)$
follows from its defining equation \eqref{eqn:phi_ei_lmd},
together with the additional assumption that
$2\lambda_i + \varepsilon \neq 0$ when $n$ is even.

By Proposition~\ref{prop:25061805}(1),
the operator
$(\phi_{\lambda}^{\lambda+\varepsilon e_i})^{N-1}$
maps $\Pi \otimes F$ into the primary component
$P_{\lambda+\varepsilon e_i}(\Pi \otimes F)$
for all sufficiently large $N$.
By assumption, the Casimir element $c_G$ acts on this primary component
by scalar multiplication.
Applying Proposition~\ref{prop:25061805}(2), we obtain
\[
   (\phi_{\lambda}^{\lambda+\varepsilon e_i})^N
   =
   \varphi_{i,\varepsilon}(\lambda)^N
   P_{\lambda+\varepsilon e_i}
   \quad \text{in $\operatorname{End}(\Pi \otimes F)$.}
\]

Substituting this identity into the conclusion of
Theorem~\ref{thm:25061808}, we obtain
\begin{align*}
  \varphi_{i,\varepsilon}(\lambda)^N
  (T \otimes \operatorname{pr}_{F \to F'})
  &\circ
  P_{\lambda+\varepsilon e_i}
  \circ (\operatorname{id}_\Pi \otimes \iota_{F'\to F}) \\
  &=
   (T \otimes \operatorname{pr}_{F \to F'})
   \circ
   (\phi_{\lambda}^{\lambda+\varepsilon e_i})^N
   \circ (\operatorname{id}_\Pi \otimes \iota_{F'\to F}) \\
  &=
   \varphi_{i,\varepsilon}(\lambda)^{N-1}
   g_{i,\varepsilon}(\lambda,\nu)\,
   T .
\end{align*}
Since $\varphi_{i,\varepsilon}(\lambda) \neq 0$,
this yields the desired formula
\[
   (T \otimes \operatorname{pr}_{F \to F'})
   \circ
   P_{\lambda+\varepsilon e_i}
   \circ (\operatorname{id}_\Pi \otimes \iota_{F'\to F})
   =
   \frac{g_{i,\varepsilon}(\lambda,\nu)}
        {\varphi_{i,\varepsilon}(\lambda)}
   \, T .
\]
\end{proof}

\section{Construction of $G'$-Invariant Elements
in $U(\mathfrak{g})$}

The proof of Theorem~\ref{thm:25061808} is completed in the following sections.
In the present section, we carry out the first step by constructing explicit
$G'$-invariant elements in the universal enveloping algebra $U(\mathfrak g)$
that arise from the diagonal action of powers of the Casimir element.

For a branching problem $G \supset G'$, the algebra $U(\mathfrak g)^{G'}$
acts naturally on the space of symmetry breaking operators.
To establish the universal scalar identities required in
Theorem~\ref{thm:25061808}, 
one is therefore led to construct
explicit $G'$-invariant elements in $U(\mathfrak g)$.
Such invariant elements are constructed in
Proposition~\ref{prop:24120139}.

\subsection{Generalities: $U(\mathfrak{g})^{\mathfrak{g}'}$
and $U(\mathfrak{g})^{G'}$}

Let $\mathfrak g_\mathbb C=\mathfrak o(n+1,\mathbb C)$ and
$\mathfrak g'_\mathbb C=\mathfrak o(n,\mathbb C)$.
As throughout the paper, all universal enveloping algebras are taken over
$\mathbb C$, and we write $U(\mathfrak g)$ for $U(\mathfrak g_\mathbb C)$.

We denote by $U(\mathfrak g)^{\mathfrak g'}$ the subalgebra of
$\mathfrak g'_\mathbb C$-invariant elements of $U(\mathfrak g)$, and by
$U(\mathfrak g)^{G'}$ the subalgebra of $G'_\mathbb C=O(n,\mathbb C)$-invariant elements.
This definition ensures that $U(\mathfrak g)^{G'}$ is independent of the choice
of real form or covering of the subgroup $G'$.

Accordingly, one has the natural inclusions
\[
   \mathfrak{Z}(G)
   \subset
   \mathfrak{Z}(\mathfrak{g}),
   \qquad
   \mathfrak{Z}(G')
   \subset
   \mathfrak{Z}(\mathfrak{g}'),
   \qquad
   U(\mathfrak{g})^{G'}
   \subset
   U(\mathfrak{g})^{\mathfrak{g}'} .
\]

The algebra $U(\mathfrak g)$ carries its standard increasing filtration
$\{U_m(\mathfrak g)\}_{m\ge0}$, where $U_m(\mathfrak g)$ is spanned by products
of at most $m$ elements of $\mathfrak g$.
The multiplication map induces, for any $a,b\in\mathbb N$, a morphism
\begin{equation}\label{eqn:Ug_mult}
   U_a(\mathfrak g)\otimes U_b(\mathfrak g)
   \longrightarrow
   U_{a+b}(\mathfrak g),
\end{equation}
which is compatible with this filtration.
All degree statements in this section are understood with respect to this
filtration.

This filtered structure will be essential in controlling the degree of
explicit $G'_\mathbb C$-invariant elements constructed below,
and in establishing the polynomial degree estimates that appear
in the universal scalar identities proved in later sections.

The following lemma summarizes the structure of $U(\mathfrak g)^{G'}$
and refines classical results of Cooper~\cite{C75} to the
group-invariant setting, including the disconnected case.

\begin{lemma}
\label{lem:25050506}
The multiplication map \eqref{eqn:Ug_mult} induces the following
commutative diagram, in which the vertical maps are isomorphisms:
\begin{equation}
\label{eqn:UUU_zUU}
\xymatrix@!C=36pt{
   \mathfrak{Z}(G) \otimes \mathfrak{Z}(G')
   \ar[d]^{\rotatebox[origin=c]{90}{$\sim$}}
   & \subset \hskip -1pc &
   \mathfrak{Z}(\mathfrak{g}) \otimes \mathfrak{Z}(\mathfrak{g}')
   \ar[d]_{\varpi}^{\rotatebox[origin=c]{90}{$\sim$}}
   \\
   U(\mathfrak{g})^{G'}
   & \subset \hskip -1pc &
   U(\mathfrak{g})^{\mathfrak{g}'}
}
\end{equation}

Moreover, these isomorphisms are compatible with the filtrations induced from the standard filtration of $U(\mathfrak g)$.
More precisely, for any $a,b \in \mathbb{N}$, the restriction of the
multiplication map induces an isomorphism
\[
   \mathfrak{Z}_a(\mathfrak{g}) \otimes \mathfrak{Z}_b(\mathfrak{g}')
   \xrightarrow{\;\sim\;}
   U(\mathfrak{g})^{\mathfrak{g}'}_{a+b}.
\]
\end{lemma}
\begin{proof}
The assertion for
$U(\mathfrak g)^{\mathfrak{g}'}$ is classical; see \cite{C75}.
The only additional point concerns the disconnectedness of the group
$G'_\mathbb C=O(n,\mathbb C)$.

We observe that the multiplication map \eqref{eqn:Ug_mult}
intertwines the diagonal action of $G_{\mathbb C}$ on
$U(\mathfrak{g}) \otimes U(\mathfrak{g})$
with the adjoint action of $G_{\mathbb C}$ on $U(\mathfrak{g})$.

Define
\begin{equation}
\label{eqn:gn}
   g_n := \operatorname{diag}(1,1,\dots,1,-1)
   \in G_{\mathbb{C}} = O(n+1,\mathbb C).
\end{equation}
Since $U(\mathfrak g)^{G'}$ coincides with the $\mathrm{Ad}(g_n)$-fixed subalgebra
of $U(\mathfrak g)^{\mathfrak g'}$, the claim follows by taking invariants.

Since $\operatorname{Ad}(g_n)$ preserves the subalgebra
$\mathfrak{g}' \subset \mathfrak{g}$, it stabilizes the
subalgebras
$\mathfrak{Z}(\mathfrak{g})$,
$\mathfrak{Z}(\mathfrak{g}')$, and
$U(\mathfrak{g})^{\mathfrak{g}'}
\subset U(\mathfrak{g})$.
Thus, the bijection
\begin{equation}
\label{eqn:zzU}
   \varpi \colon
   \mathfrak{Z}(\mathfrak{g}) \otimes \mathfrak{Z}(\mathfrak{g}')
   \longrightarrow
   U(\mathfrak{g})^{\mathfrak{g}'}
\end{equation}
is $\operatorname{Ad}(g_n)$-equivariant.

We now distinguish the argument according to the parity of $n$.

If $n$ is even, then $\mathfrak{Z}(G)=\mathfrak{Z}(\mathfrak{g})$.
Hence, the bijection $\varpi$ in \eqref{eqn:zzU} induces an isomorphism
\begin{equation}
\label{eqn:UzU}
   \mathfrak{Z}(G)
   \otimes
   \mathfrak{Z}(\mathfrak{g}')
   \xrightarrow{\;\sim\;}
   U(\mathfrak{g})^{\mathfrak{g}'}.
\end{equation}
Moreover, $\operatorname{Ad}(g_n)$ acts trivially on
$\mathfrak{Z}(G)$.
Taking $\operatorname{Ad}(g_n)$-fixed vectors,
the isomorphism \eqref{eqn:UzU} yields
the left-hand vertical bijection in
\eqref{eqn:UUU_zUU}.

If $n$ is odd, then $\mathfrak{Z}(G')=\mathfrak{Z}(\mathfrak{g}')$.
In this case, $\varpi$ induces an isomorphism
\begin{equation}
\label{eqn:zUU}
   \mathfrak{Z}(\mathfrak{g}) \otimes \mathfrak{Z}(G')
   \xrightarrow{\;\sim\;}
   U(\mathfrak{g})^{G'}.
\end{equation}
Since $\operatorname{Ad}(g_n)$ acts trivially on $\mathfrak{Z}(G')$,
taking $\operatorname{Ad}(g_n)$-invariants in
\eqref{eqn:zUU} yields
the left-hand vertical bijection in
\eqref{eqn:UUU_zUU}.

The compatibility with the filtrations follows immediately from the
definition of the standard filtration and the multiplicativity property \eqref{eqn:Ug_mult}.
\end{proof}
Lemma~\ref{lem:25050506} yields the following generalization of the
Harish-Chandra isomorphism.
Let
\begin{equation}
\label{eqn:WG}
W_G = \mathfrak S_r \ltimes (\mathbb{Z}/2\mathbb{Z})^r,
\qquad
W_{G'} = \mathfrak S_s \ltimes (\mathbb{Z}/2\mathbb{Z})^s,
\end{equation}
where $r=\left\lfloor \frac{n+1}{2}\right\rfloor$ and
$s=\left\lfloor \frac{n}{2}\right\rfloor$ are the ranks of $G$ and $G'$,
respectively.

\begin{lemma}
\label{lem:zzU}
For any pair $(\lambda,\nu)$ belonging either to
\[
   \operatorname{Hom}_{\mathbb{C}\text{-alg}}
   \bigl(\mathfrak{Z}(\mathfrak g), \mathbb C\bigr)
   \times
   \operatorname{Hom}_{\mathbb{C}\text{-alg}}
   \bigl(\mathfrak{Z}(\mathfrak g'), \mathbb C\bigr)
   \;\simeq\;
   \mathbb C^r/W_{\mathfrak g} \times \mathbb C^s/W_{\mathfrak g'},
\]
or to
\[
   \operatorname{Hom}_{\mathbb{C}\text{-alg}}
   \bigl(\mathfrak{Z}(G), \mathbb C\bigr)
   \times
   \operatorname{Hom}_{\mathbb{C}\text{-alg}}
   \bigl(\mathfrak{Z}(G'), \mathbb C\bigr)
   \;\simeq\;
   \mathbb C^r/W_G \times \mathbb C^s/W_{G'},
\]
there exists a unique algebra homomorphism
\[
   \chi_{\lambda,\nu} \colon
   U(\mathfrak{g})^{\mathfrak{g}'}
   \to \mathbb{C},
   \qquad\text{or respectively}\qquad
   \chi_{\lambda,\nu} \colon
   U(\mathfrak{g})^{G'}
   \to \mathbb{C},
\]
such that the following diagram is commutative, where the left-hand
(resp.\ right-hand) square corresponds to the Lie algebra
(resp.\ group) invariant case:
\[
\xymatrix{
   {\mathfrak{Z}}(\mathfrak{g}) \otimes {\mathfrak{Z}}(\mathfrak g')
     \ar[d]^{\rotatebox[origin=c]{90}{$\sim$}}
     \ar[rd]
     & & 
   \mathfrak{Z}(G) \otimes \mathfrak{Z}(G')
     \ar[d]^{\rotatebox[origin=c]{90}{$\sim$}}
     \ar[rd]
\\
   U(\mathfrak{g})^{\mathfrak{g}'}
     \ar[r]_(.6){\chi_{\lambda,\nu}}
     & \mathbb{C}
     &
   U(\mathfrak{g})^{G'}
     \ar[r]_(.6){\chi_{\lambda,\nu}}
     & \mathbb{C}
}
\]
\end{lemma}
\begin{proposition}
\label{prop:on_ring_onto}
Let $G \supset G'$ be a real form of
$(O(n+1,\mathbb{C}), O(n,\mathbb{C}))$, or a finite covering thereof.
Let $\Pi \in \mathcal{M}(G)$ have
$\mathfrak{Z}(G)$-infinitesimal character $\lambda$,
and let $\pi \in \mathcal{M}(G')$ have
$\mathfrak{Z}(G')$-infinitesimal character $\nu$.
Then the algebra $U(\mathfrak g)^{G'}$ acts on the space
$\operatorname{Hom}_{G'}(\Pi|_{G'}, \pi)$
via the character $\chi_{\lambda,\nu}$.
More precisely, for any symmetry breaking operator
$T \colon \Pi \to \pi$, one has
\[
   T\bigl(d\Pi(D)u\bigr)
   =
   \chi_{\lambda,\nu}(D)\, T(u),
   \qquad
   u \in \Pi,\ D \in U(\mathfrak{g})^{G'}.
\]
\end{proposition}

\begin{proof}
By Lemma~\ref{lem:25050506},
we may write
$
   D = \varpi\Bigl(\sum_i z_i \otimes z_i'\Bigr)
$
for some $z_i \in \mathfrak{Z}(G)$
and $z_i' \in \mathfrak{Z}(G')$.
Since $T$ intertwines the $G'$-actions, it follows that
\[
   T\bigl(d\Pi(D)u\bigr)
   =
   \sum_i d \pi(z_i')\, T\bigl(d\Pi(z_i)u\bigr)
   =
   \sum_i \chi_{\lambda}(z_i)\,
           \chi_{\nu}(z_i')\, T(u).
\]
By the definition of $\chi_{\lambda,\nu}$, this equals
$\chi_{\lambda,\nu}(D)\, T(u)$.
\end{proof}
\subsection{Constructing Explicit Elements in $U(\mathfrak{g})^{G'}$}

As shown in Lemma~\ref{lem:25050506}, every element of
$U(\mathfrak{g})^{G'}$ can be expressed as a linear combination of
elements of the form $uv$, where $u \in \mathfrak{Z}(G)$ and
$v \in \mathfrak{Z}(G')$.
In concrete situations, however, it is often not transparent how to
verify whether a given element of $U(\mathfrak{g})$ actually belongs to
$U(\mathfrak{g})^{G'}$.

The main result of this subsection is
Proposition~\ref{prop:24120139}, which provides an explicit construction
of elements in $U(\mathfrak{g})^{G'}$ without recourse to the centers
$\mathfrak{Z}(G)$ or $\mathfrak{Z}(G')$.
This proposition plays a crucial role in the proof of
Theorem~\ref{thm:25061808}.

Recall from \eqref{eqn:Xij} that
$X_{ij} = E_{ij} - E_{ji} \in \mathfrak{g}_{\mathbb{C}}
= \mathfrak{o}(n+1,\mathbb{C})$ for $0 \le i \neq j \le n$.
With our conventions,
$X_{ij} \in \mathfrak{g}'_{\mathbb{C}} = \mathfrak{o}(n,\mathbb{C})$
whenever $1 \le i \neq j \le n$.

We introduce elements
$A^{(N)}, B_j^{(N)} \in U_N(\mathfrak{g})$
$(1 \le j \le n)$,
defined inductively by
\[
   A^{(1)} := 0,
   \qquad
   B_j^{(1)} := X_{0j}.
\]
\begin{align}
\label{eqn:ANdef}
   A^{(N+1)}
   & :=
   \sum_{k=1}^{n} X_{k0}\, B_k^{(N)}, \\
\label{eqn:BNjdef}
   B_j^{(N+1)}
   & :=
   X_{0j}\, A^{(N)}
   +
   \sum_{k=1}^{n} X_{kj}\, B_k^{(N)} .
\end{align}
These elements naturally occur in the analysis of powers of the Casimir element and will be essential in Proposition~\ref{prop:24113019}. For small values of $N$, they take the following explicit forms.
\begin{example}
\label{ex:24113021}
For small values of $N$, the elements $A^{(N)}$ and $B_j^{(N)}$ defined above
are given explicitly as follows:
\begin{align*}
   A^{(2)}
   &= \sum_{a=1}^{n} X_{a0} X_{0a},
\\
   B_j^{(2)}
   &= \sum_{a=1}^{n} X_{aj} X_{0a},
\\
   A^{(3)}
   &= \sum_{a=1}^{n} \sum_{b=1}^{n}
      X_{b0} X_{ab} X_{0a},
\\
   B_j^{(3)}
   &= X_{0j} \sum_{a=1}^{n} X_{a0} X_{0a}
      + \sum_{a=1}^{n} \sum_{b=1}^{n}
        X_{bj} X_{ab} X_{0a},
\\
   A^{(4)}
   &= \biggl( \sum_{b=1}^{n} X_{b0} X_{0b} \biggr)
      \biggl( \sum_{a=1}^{n} X_{a0} X_{0a} \biggr)
      + \sum_{a=1}^{n} \sum_{b=1}^{n} \sum_{c=1}^{n}
        X_{c0} X_{bc} X_{ab} X_{0a}.
%
\end{align*}
\end{example}

We now set $D_j^{(1)} := X_{j0}$.
For $\ell \in \mathbb{N}_{+}$, we define elements
$C^{(\ell+1)}$, $D_j^{(\ell)}$, and $\mathcal{D}^{\ell,N}$ inductively as follows:
\begin{align}
   C^{(\ell+1)}
   & :=
   \sum_{j_1=1}^{n} \cdots \sum_{j_{\ell}=1}^{n}
   X_{j_{\ell}0} \cdots X_{j_1 j_2} X_{0 j_1}
   \in U_{\ell+1}(\mathfrak{g}),
\notag
\\
\label{eqn:dlj}
   D_j^{(\ell)}
   & :=
   \sum_{i_1=1}^{n} \cdots \sum_{i_{\ell-1}=1}^{n}
   X_{i_{\ell-1}0} \cdots X_{i_2 i_3}
   X_{i_1 i_2} X_{j i_1}
   \in U_{\ell}(\mathfrak{g}),
\\
\label{eqn:24120141}
   \mathcal{D}^{\ell,N}
   & :=
   \sum_{j=1}^{n} D_j^{(\ell)} B_j^{(N)}
   \in U_{\ell+N}(\mathfrak{g}).
\end{align}

These definitions admit a convenient symbolic interpretation in terms of
directed paths with $\ell$ arrows:
\begin{align*}
   & C^{(\ell)}
     \colon
     0 \leftarrow \bullet \leftarrow \bullet \leftarrow \bullet
     \cdots \bullet \leftarrow 0,
\\
   & D_j^{(\ell)}
     \colon
     0 \leftarrow \bullet \leftarrow \bullet \leftarrow \bullet
     \cdots \bullet \leftarrow j .
\end{align*}
The main result of this section is the following proposition, whose proof
will be given in the next subsection.

\begin{proposition}
\label{prop:24120139}
For any $\ell \in \mathbb{N}_{+}$,
the elements $A^{(\ell)}$, $C^{(\ell+1)}$, and
$\mathcal{D}^{\ell,N}$
belong to $U(\mathfrak{g})^{G'}$.
\end{proposition}

\subsection{Proof of Proposition~\ref{prop:24120139}}
\subsection{Proof of Proposition~\ref{prop:24120139}}

This subsection is devoted to the proof of
Proposition~\ref{prop:24120139}.

We first record several identities relating
$A^{(\ell)}$, $C^{(\ell+1)}$, and $\mathcal{D}^{\ell,N}$.

\begin{lemma}
\label{lem:ACD_relation}
\rm
\begin{align*}
C^{(\ell+1)}
  &= \mathcal{D}^{\ell,1}
   = \sum_{j=1}^{n} D_j^{(\ell)} X_{0j}, \\
D_k^{(\ell+1)}
  &= \sum_{j=1}^{n} D_j^{(\ell)} X_{kj}, \\
\mathcal{D}^{\ell,N}
  &= C^{(\ell+1)}A^{(N-1)}+\mathcal{D}^{\ell+1,N-1}.
\end{align*}
\end{lemma}

\begin{proof}
The first two identities follow directly from the definitions of
$C^{(\ell+1)}$, $D_j^{(\ell)}$, and $\mathcal{D}^{\ell,N}$.

\medskip
To prove the third identity, we substitute \eqref{eqn:BNjdef} into
\eqref{eqn:24120141} to obtain
\begin{align*}
   \mathcal{D}^{\ell,N}
   &=
   \sum_{j=1}^{n} D_j^{(\ell)} X_{0j} A^{(N-1)}
   + \sum_{j=1}^{n} \sum_{k=1}^{n} D_j^{(\ell)} X_{kj} B_k^{(N-1)} \\
   &=
   C^{(\ell+1)} A^{(N-1)}
   + \sum_{k=1}^{n}
     \Bigl( \sum_{j=1}^{n} D_j^{(\ell)} X_{kj} \Bigr) B_k^{(N-1)}.
\end{align*}
By the second identity, the inner sum equals $D_k^{(\ell+1)}$, and therefore
\[
   \mathcal{D}^{\ell,N}
   = C^{(\ell+1)} A^{(N-1)}
     + \sum_{k=1}^{n} D_k^{(\ell+1)} B_k^{(N-1)}
   = C^{(\ell+1)} A^{(N-1)} + \mathcal{D}^{\ell+1,N-1}.
\]
This completes the proof.
\end{proof}

In the proof below, 
 we frequently use the following basic formula
 for the Lie bracket in ${\mathfrak{o}}_{n+1}$:
\begin{equation}
\label{eqn:XX}
 [X_{ab}, X_{ij}]
 =\delta_{b i}X_{a j}+\delta_{b j}X_{i a}+\delta_{a i}X_{j b}+\delta_{a j}X_{b i}.  
\end{equation}

\begin{lemma}
\label{lem:250426}
Let $1 \le a, b \le n$,
$1 \le j \le n$,
and $\ell \in \mathbb{N}_{+}$.
Then the following commutation relations hold:
\begin{align}
\label{eqn:25042614}
[X_{ab}, A^{(\ell)}]
&= 0,
\\
\label{eqn:25042610}
[X_{ab}, B_j^{(\ell)}]
&= \delta_{bj} B_a^{(\ell)} - \delta_{aj} B_b^{(\ell)},
\\
\label{eqn:25042617}
[X_{ab}, D_j^{(\ell)}]
&= -\delta_{aj} D_b^{(\ell)} + \delta_{bj} D_a^{(\ell)},
\\
\label{eqn:25042619}
[X_{ab}, \mathcal{D}^{\ell, N}]
&= 0.
\end{align}
\end{lemma}
\begin{proof}
We divide the proof into several steps.

We begin with the simplest case, namely the identity
\eqref{eqn:25042617} for $D_j^{(\ell)}$, which is proved in
\textbf{Step~1}.
The identity \eqref{eqn:25042610} for $B_j^{(\ell)}$ is proved by a
double induction: first
\[
   A^{(\ell-1)} \ \text{and}\ B_j^{(\ell-1)}
   \ \Longrightarrow\ B_j^{(\ell)}
\]
in \textbf{Step~2}, and then
\[
   B_j^{(\ell)} \ \Longrightarrow\ A^{(\ell)}
\]
in \textbf{Step~3}.
Finally, the remaining assertion \eqref{eqn:25042619} is proved in
\textbf{Step~4} using \eqref{eqn:25042610} and \eqref{eqn:25042617}.

\medskip
\noindent\textbf{Step~1.}
We prove \eqref{eqn:25042617} for $D_j^{(\ell)}$ by induction on $\ell$.

For $\ell=1$, we have $D_j^{(1)} = X_{j0}$.
It follows immediately from \eqref{eqn:XX} that
\[
   [X_{ab}, D_j^{(1)}]
   = \delta_{bj} X_{a0} - \delta_{aj} X_{b0}
   = -\delta_{aj} D_b^{(1)} + \delta_{bj} D_a^{(1)} .
\]

Assume that \eqref{eqn:25042617} holds for $D_j^{(\ell)}$.
From the definition \eqref{eqn:dlj}, we have
\[
   D_k^{(\ell+1)}
   = \sum_{j=1}^{n} D_j^{(\ell)} X_{kj}.
\]
Therefore,
\[
   [X_{ab}, D_k^{(\ell+1)}]
   =
   \sum_{j=1}^{n} [X_{ab}, D_j^{(\ell)}] X_{kj}
   +
   \sum_{j=1}^{n} D_j^{(\ell)} [X_{ab}, X_{kj}].
\]

By the induction hypothesis and \eqref{eqn:XX}, the right-hand side equals
\[
   \sum_{j=1}^{n}
   \bigl(-\delta_{aj} D_b^{(\ell)} + \delta_{bj} D_a^{(\ell)}\bigr) X_{kj}
   +
   \sum_{j=1}^{n} D_j^{(\ell)}
   (\delta_{bk} X_{aj} + \delta_{bj} X_{ka}
    + \delta_{ak} X_{jb} + \delta_{aj} X_{bk}).
\]
Collecting terms, this simplifies to
\[
   \delta_{bk} \sum_{j=1}^{n} D_j^{(\ell)} X_{aj}
   -
   \delta_{ak} \sum_{j=1}^{n} D_j^{(\ell)} X_{bj}
   =
   \delta_{bk} D_a^{(\ell+1)}
   -
   \delta_{ak} D_b^{(\ell+1)}.
\]
Thus \eqref{eqn:25042617} holds for $D_j^{(\ell+1)}$, completing the inductive step.

\noindent{\bf Step~2.}\enspace
We prove \eqref{eqn:25042610} for $B_j^{(\ell+1)}$ under the assumptions
that \eqref{eqn:25042614} holds for $A^{(\ell)}$ and
\eqref{eqn:25042610} holds for $B_j^{(\ell)}$.

By the definition \eqref{eqn:BNjdef} of $B_j^{(\ell+1)}$, we have
\begin{align*}
 [X_{ab}, B_j^{(\ell+1)}]
 =\,
 & [X_{ab}, X_{0j}]\,A^{(\ell)}
   + X_{0j}[X_{ab}, A^{(\ell)}]
\\
 & + \sum_{k=1}^n [X_{ab}, X_{kj}]\,B_k^{(\ell)}
   + \sum_{k=1}^n X_{kj}[X_{ab}, B_k^{(\ell)}].
\end{align*}
By the inductive hypothesis, the second term vanishes, and the fourth
term equals $X_{bj}B_a^{(\ell)} - X_{aj}B_b^{(\ell)}$.

On the other hand, a direct computation using \eqref{eqn:XX} gives
\begin{align*}
 \text{(first term)}
 &= \delta_{bj} X_{0a} A^{(\ell)}
    + \delta_{aj} X_{b0} A^{(\ell)},
\\
 \text{(third term)}
 &= X_{aj} B_b^{(\ell)} + X_{jb} B_a^{(\ell)}
    + \delta_{bj} \sum_{k=1}^n X_{ka} B_k^{(\ell)}
    + \delta_{aj} \sum_{k=1}^n X_{bk} B_k^{(\ell)}.
\end{align*}
Combining these expressions, we obtain
\begin{align*}
 [X_{ab}, B_j^{(\ell+1)}]
 =\,
 & \delta_{aj}\Bigl(X_{b0} A^{(\ell)}
   + \sum_{k=1}^n X_{bk} B_k^{(\ell)}\Bigr)
\delta_{bj}\Bigl(X_{0a} A^{(\ell)}
   + \sum_{k=1}^n X_{ka} B_k^{(\ell)}\Bigr)
\\
 =\,
 & -\delta_{aj} B_b^{(\ell+1)}
   + \delta_{bj} B_a^{(\ell+1)},
\end{align*}
where the last equality follows from the definition
\eqref{eqn:BNjdef} of $B_a^{(\ell+1)}$.

This completes Step~2.

\noindent{\bf Step~3.}\enspace
We prove \eqref{eqn:25042614} for $A^{(\ell)}$ and
\eqref{eqn:25042610} for $B_j^{(\ell)}$ by induction on $\ell$.

\medskip
\noindent
We first verify the assertions for $\ell=1$.

Since $A^{(1)}=0$, the relation \eqref{eqn:25042614} for $A^{(\ell)}$
is trivial when $\ell=1$.
Moreover, since $B_j^{(1)}=X_{0j}$, it follows from \eqref{eqn:XX} that
\[
   [X_{ab}, B_j^{(1)}]
   = \delta_{bj} X_{a0} - \delta_{aj} X_{b0}
   = \delta_{bj} B_a^{(1)} - \delta_{aj} B_b^{(1)},
\]
which proves \eqref{eqn:25042610} for $B_j^{(\ell)}$ when $\ell=1$.

\medskip
\noindent
Next, assume that the assertions of Step~3 hold for $\ell-1$.
Then \eqref{eqn:25042610} for $B_j^{(\ell)}$ has already been established
in Step~2.

It remains to prove \eqref{eqn:25042614} for $A^{(\ell)}$.
By the definition \eqref{eqn:ANdef} of $A^{(\ell)}$, we have
\[
   [X_{ab}, A^{(\ell)}]
   =
   \sum_{k=1}^{n} [X_{ab}, X_{k0}]\, B_k^{(\ell-1)}
   +
   \sum_{k=1}^{n} X_{k0}\,[X_{ab}, B_k^{(\ell-1)}].
\]

Using \eqref{eqn:XX} together with the inductive hypothesis
\eqref{eqn:25042610} for $B_k^{(\ell-1)}$, we compute
\begin{align*}
   \text{(first term)}
   &=
   \sum_{k=1}^{n}
   (\delta_{bk} X_{a0} + \delta_{ak} X_{0b}) B_k^{(\ell-1)}
   =
   X_{a0} B_b^{(\ell-1)} + X_{0b} B_a^{(\ell-1)}, 
\\
   \text{(second term)}
   &=
   \sum_{k=1}^{n}
   X_{k0} (\delta_{bk} B_a^{(\ell-1)} - \delta_{ak} B_b^{(\ell-1)})
   =
   X_{b0} B_a^{(\ell-1)} - X_{a0} B_b^{(\ell-1)}.
\end{align*}
Therefore, the two terms cancel each other, and we conclude that
\[
   [X_{ab}, A^{(\ell)}]=0.
\]
This completes Step~3.

\noindent{\bf Step~4.}\enspace
We prove \eqref{eqn:25042619} for $\mathcal{D}^{\ell,N}$
using \eqref{eqn:25042610} for $B_j^{(N)}$
and \eqref{eqn:25042617} for $D_j^{(\ell)}$.
By the definition \eqref{eqn:24120141} of $\mathcal{D}^{\ell,N}$, we have
\[
   [X_{ab}, \mathcal{D}^{\ell,N}]
   =
   \sum_{j=1}^{n} [X_{ab}, D_j^{(\ell)}]\, B_j^{(N)}
   +
   \sum_{j=1}^{n} D_j^{(\ell)} [X_{ab}, B_j^{(N)}].
\]
Using \eqref{eqn:25042617} for $[X_{ab}, D_j^{(\ell)}]$
and \eqref{eqn:25042610} for $[X_{ab}, B_j^{(N)}]$, we obtain
\begin{align*}
   [X_{ab}, \mathcal{D}^{\ell,N}]
   =
   \sum_{j=1}^{n}
   (-\delta_{aj} D_b^{(\ell)} + \delta_{bj} D_a^{(\ell)}) B_j^{(N)}
   +
   \sum_{j=1}^{n}
   D_j^{(\ell)} (\delta_{bj} B_a^{(N)} - \delta_{aj} B_b^{(N)}).
\end{align*}
Simplifying, this becomes
\[
   (-D_b^{(\ell)} B_a^{(N)} + D_a^{(\ell)} B_b^{(N)})
   +
   (D_b^{(\ell)} B_a^{(N)} - D_a^{(\ell)} B_b^{(N)})
   = 0.
\]
This completes the proof.
\end{proof}

Note that the subalgebra of
$\operatorname{Ad}(g_n)$-invariant elements in
$U(\mathfrak{g})^{\mathfrak{g}'}$
coincides with $U(\mathfrak{g})^{G'}$,
where 
$
   g_n = \operatorname{diag}(1,1,\dots,1,-1)
   \in G'_\mathbb{C} \subset G_\mathbb{C}
$
is as in \eqref{eqn:gn}.
We now show that
$A^{(\ell)}$, $C^{(\ell+1)}$, and $\mathcal{D}^{\ell,N}$
are invariant under $\operatorname{Ad}(g_n)$.

\begin{lemma}
\label{lem:25050414}
For any $\ell \in \mathbb{N}_{+}$, one has
\begin{align}
\label{eqn:al_sgn}
\operatorname{Ad}(g_n) A^{(\ell)}
&= A^{(\ell)}, 
\\
\label{eqn:bl_sgn}
\operatorname{Ad}(g_n) B_j^{(\ell)}
&=
\begin{cases}
 B_j^{(\ell)} & \text{if $j \neq n$,} \\
 -B_j^{(\ell)} & \text{if $j = n$,}
\end{cases}
\\
\label{eqn:dl_sgn}
\operatorname{Ad}(g_n) D_j^{(\ell)}
&=
\begin{cases}
 D_j^{(\ell)} & \text{if $j \neq n$,} \\
 -D_j^{(\ell)} & \text{if $j = n$,}
\end{cases}
\\
\notag
\operatorname{Ad}(g_n) \mathcal{D}^{\ell,N}
&= \mathcal{D}^{\ell,N},
\qquad
\operatorname{Ad}(g_n) C^{(\ell+1)}
= C^{(\ell+1)} .
\end{align}
\end{lemma}

\begin{proof}
Using the formula
\[
   \operatorname{Ad}(g_n) X_{ij}
   =
   \begin{cases}
   -X_{ij} & \text{if $i=n$ or $j=n$,} \\
   \;\,X_{ij} & \text{if $0 \le i,j \le n-1$,}
   \end{cases}
\]
the identity \eqref{eqn:dl_sgn} for $D_j^{(\ell)}$
follows directly from its definition \eqref{eqn:dlj}.

We next verify \eqref{eqn:al_sgn} and \eqref{eqn:bl_sgn}
for $A^{(\ell)}$ and $B_j^{(\ell)}$
by induction on $\ell$.
For $\ell=1$, the claims hold since
$A^{(1)}=0$ and $B_j^{(1)}=X_{0j}$.

Assume that \eqref{eqn:al_sgn} and \eqref{eqn:bl_sgn}
hold for a given $\ell$.
Rewriting \eqref{eqn:ANdef}, we have
\[
   A^{(\ell+1)}
   =
   X_{n0} B_n^{(\ell)}
   + \sum_{k=1}^{n-1} X_{k0} B_k^{(\ell)} .
\]
Then \eqref{eqn:al_sgn} for $A^{(\ell+1)}$
follows immediately from \eqref{eqn:bl_sgn}.

Similarly, rewriting \eqref{eqn:BNjdef} yields
\[
   B_j^{(\ell+1)}
   =
   X_{0j} A^{(\ell)}
   + X_{nj} B_n^{(\ell)}
   + \sum_{k=1}^{n-1} X_{kj} B_k^{(\ell)} .
\]
The transformation rule \eqref{eqn:bl_sgn}
for $B_j^{(\ell+1)}$ then follows from
\eqref{eqn:al_sgn} and \eqref{eqn:bl_sgn}
for $A^{(\ell)}$ and $B_j^{(\ell)}$.
Thus \eqref{eqn:al_sgn} and \eqref{eqn:bl_sgn}
hold for all $\ell$.

Finally, rewriting \eqref{eqn:24120141} as
\[
   \mathcal{D}^{\ell,N}
   =
   D_n^{(\ell)} B_n^{(N)}
   + \sum_{j=1}^{n-1} D_j^{(\ell)} B_j^{(N)},
\]
the invariance of $\mathcal{D}^{\ell,N}$
follows immediately from \eqref{eqn:bl_sgn}
and \eqref{eqn:dl_sgn}.
\end{proof}

\begin{proof}
[Proof of Proposition~\ref{prop:24120139}]
By Lemma~\ref{lem:250426} together with the first formula in Lemma~\ref{lem:ACD_relation},
we have
\[
   A^{(\ell)},\; C^{(\ell+1)},\; \mathcal{D}^{\ell,N}
   \in U(\mathfrak{g})^{\mathfrak{g}'}.
\]
Lemma~\ref{lem:25050414} then implies that
\[
   A^{(\ell)},\; C^{(\ell+1)},\; \mathcal{D}^{\ell,N}
   \in U(\mathfrak{g})^{G'}.
\]
This completes the proof of Proposition~\ref{prop:24120139}.
\end{proof}

\section{Composition Formula for Symmetry Breaking Operators}
This section establishes the essential part of the proof of the universal
scalar identity (Theorem~\ref{thm:25061808}) by expressing the relevant
symmetry breaking operators as explicit polynomials in the Casimir element.
The determination of the resulting rational functions is carried out in the
final section.

More precisely, we introduce and analyze the composition
\[
 (T \otimes \operatorname{pr}_{F \to F'})
 \circ P_{\lambda + \varepsilon e_i}
 \circ (\operatorname{id} \otimes \iota_{F'\to F}),
\]
and derive an explicit formula for this expression using a polynomial in the Casimir element.

\subsection{Diagonal Action of Powers of the Casimir Element}

We analyze and compute the diagonal action of powers of the Casimir element
on the tensor product representation $\Pi \otimes F$.
The main result of this subsection is
Proposition~\ref{prop:24113019}, which serves as a foundational step
in the proof of Theorem~\ref{thm:25061808}.

Let $\mathfrak g_\mathbb c=\mathfrak{o}(n+1, \mathbb C)$ and $r=\lfloor \frac{n+1}{2} \rfloor$.
We fix $\lambda \in \mathbb{C}^r$.
It is convenient to introduce a shifted Casimir operator,
depending on $\lambda$, by
\begin{equation}
\label{eqn:cGtilde}
   \widetilde c_G
   :=
   \frac{1}{2}
   \bigl(
      c_G
      -
      (\lVert \lambda \rVert^2 - \lVert \rho_\mathfrak{g} \rVert^2 + n)
   \bigr).
\end{equation}

\medskip
\noindent
Recall that $F=\mathbb{C}^{n+1}$ denotes the standard representation
of $G$, with standard basis
\[
   f_0, f_1, \dots, f_n \in F.
\]
The one-dimensional subspace
$
   F'=\mathbb{C} f_0
$
consists of $G'$-fixed vectors.
As in \eqref{eqn:pr_iota}, we write
\[
   \operatorname{pr}_{F \to F'}
   \colon F \to \mathbb{C},
   \qquad
   \iota_{F' \to F}
   \colon \mathbb{C} \to F
\]
for the natural projection and embedding, respectively.
\medskip
We also recall that the elements
$A^{(N)}$ and $B_j^{(N)} \in U_N(\mathfrak{g})$
$(1 \le j \le n)$
are defined inductively in
\eqref{eqn:ANdef} and \eqref{eqn:BNjdef}, respectively.

\begin{proposition}
\label{prop:24113019}
Suppose that $\Pi \in \mathcal{M}(G)$ has
$\mathfrak{Z}(G)$-infinitesimal character $\lambda$.
For any vector $u \in \Pi$, the following identity holds in
$\Pi \otimes F$:
\begin{equation}
\label{eqn:24113017}
   \widetilde c_G^{\,N} (u \otimes f_0)
   =
   (A^{(N)} u) \otimes f_0
   +
   \sum_{j=1}^{n} B_j^{(N)} u \otimes f_j .
\end{equation}
\end{proposition}

The proof of Proposition~\ref{prop:24113019} repeatedly uses the following lemma.

\begin{lemma}
\label{lem:24113017}
For any vectors $u, u_1, \dots, u_n \in \Pi$, one has the following
identities in $\Pi \otimes F$:
\begin{align*}
   \widetilde c_G (u \otimes f_0)
   &=
   \Bigl(\sum_{j=1}^{n} X_{0j} u \Bigr) \otimes f_j,
\\
   \widetilde c_G \Bigl(\sum_{j=1}^{n} u_j \otimes f_j \Bigr)
   &=
   \Bigl(\sum_{k=1}^{n} X_{k0} u_k \Bigr) \otimes f_0
   +
   \sum_{j=1}^{n}
   \Bigl(\sum_{i=1}^{n} X_{ij} u_i \Bigr) \otimes f_j .
\end{align*}
\end{lemma}

\begin{proof}
The Casimir element
\[
   c_G = -\sum_{0 \le i < j \le n} X_{ij}^2
\]
acts on $\Pi$ and on $F=\mathbb{C}^{n+1}$
by scalar multiplication with
$
 \lVert \lambda \rVert^2 - \lVert \rho_{\mathfrak g} \rVert^2
$
and $n$, respectively.
By the Leibniz rule, 
the shifted Casimir element
$ 
   \widetilde c_G = \frac{1}{2}
    \bigl(
      c_G
      -
      (\lVert \lambda \rVert^2 - \lVert \rho_\mathfrak{g}\rVert^2 + n)
   \bigr)
$
acts on $u \otimes f$ by
\[
   \widetilde c_G(u \otimes f)
   =
   \sum_{0 \le i < j \le n} X_{ij} u \otimes X_{ji} f .
\]

With our conventions, the standard representation of
$\mathfrak{o}(n+1, \mathbb C)$ on $\mathbb{C}^{n+1}$ is given by
\begin{equation}
\label{eqn:on_Cn}
   X_{ki} f_j
   =
   \delta_{ij} f_k - \delta_{kj} f_i .
\end{equation}
The first identity follows immediately from \eqref{eqn:on_Cn} with $j=0$.

For the second identity, we compute
\begin{align*}
   \widetilde c_G\Bigl(\sum_{j=1}^{n} u_j \otimes f_j \Bigr)
   &=
   \sum_{j=1}^{n} \sum_{0 \le i < k \le n}
   X_{ik} u_j \otimes X_{ki} f_j
\\
   &=
   \sum_{0 \le i < k \le n}
   \bigl(
      X_{ik} u_i \otimes f_k
      -
      X_{ik} u_k \otimes f_i
   \bigr)
\\
   &=
   \Bigl(\sum_{k=1}^{n} X_{k0} u_k \Bigr) \otimes f_0
   +
   \sum_{j=1}^{n}
   \Bigl(\sum_{i=1}^{n} X_{ij} u_i \Bigr) \otimes f_j .
\end{align*}
This completes the proof.
\end{proof}

\begin{proof}
[Proof of Proposition~\ref{prop:24113019}]
We argue by induction on $N$.

For $N=1$, we have $A^{(1)}=0$ and $B_j^{(1)}=X_{0j}$.
Hence, the assertion follows immediately from
Lemma~\ref{lem:24113017}.

Assume that \eqref{eqn:24113017} holds for a given integer $N$.
Applying Lemma~\ref{lem:24113017} to the vector
\[
   \widetilde c_G^{\,N}(u \otimes f_0),
\]
and using the recurrence relations
\eqref{eqn:ANdef} and \eqref{eqn:BNjdef} defining
$A^{(N+1)}$ and $B_j^{(N+1)}$, we obtain
\[
   \widetilde c_G^{\,N+1} (u \otimes f_0)
   =
   (A^{(N+1)} u) \otimes f_0
   +
   \sum_{j=1}^{n} (B_j^{(N+1)} u) \otimes f_j,
\]
which is exactly \eqref{eqn:24113017} for $N+1$.
This completes the inductive step, and hence the proof.
\end{proof}

Let $\chi_{\lambda,\nu}\colon U(\mathfrak{g})^{G'}\to \mathbb C$ be the algebra homomorphism defined in
Lemma~\ref{lem:zzU}.
Let $A^{(\ell)} \in U(\mathfrak{g})^{G'}$ be as defined in
\eqref{eqn:ANdef}.
For each $\ell \in \mathbb{N}_{+}$, we define a polynomial
$b^{(\ell)}(\lambda,\nu)$ by
\begin{equation}
\label{eqn:bN_lmdnu}
   b^{(\ell)}(\lambda,\nu)
   :=
   \chi_{\lambda,\nu}\!\left(A^{(\ell)}\right).
\end{equation}

\begin{proposition}
\label{prop:TprCA}
Let $\Pi\in\mathcal{M}(G)$ be a representation with
$\mathfrak{Z}(G)$-infinitesimal character $\lambda$, 
let
$\pi\in\mathcal{M}(G')$ be a representation with
$\mathfrak{Z}(G')$-infinitesimal character $\nu$,
and let $T\colon \Pi \to \pi$ be a symmetry breaking operator.
Then one has
\[
(T \otimes \operatorname{pr}_{F \to F'})
\bigl(\widetilde c_G^{\,\ell} (u \otimes f_0)\bigr)
=
b^{(\ell)}(\lambda,\nu)\, (T(u)) \otimes f_0
\]
for every $u \in \Pi$ and $\ell \in \mathbb{N}_{+}$.
\end{proposition}

\begin{proof}
By Proposition~\ref{prop:24113019}, we have
\[
   (T \otimes \operatorname{pr}_{F\to F'})
   \bigl(\widetilde c_G^{\,\ell}(u \otimes f_0)\bigr)
   =
   T\bigl(A^{(\ell)}u\bigr)\otimes f_0.
\]
As shown in Proposition~\ref{prop:24120139},
the element $A^{(\ell)}$ belongs to
$U(\mathfrak{g})^{G'}$.
The asserted identity therefore follows from
Proposition~\ref{prop:on_ring_onto}.
\end{proof}

Before describing the general properties of $b^{(\ell)}$, we give explicit formulas for $b^{(\ell)}(\lambda,\nu)$
for $\ell=1,2,3$.

\begin{example}
\label{ex:blmdnu}
\noindent{\rm(1)}\enspace
$\ell=1$.
Since $A^{(1)}=0$, we have
\[
   b^{(1)}(\lambda,\nu)=0.
\]

\smallskip
\noindent{\rm(2)}\enspace
$\ell=2$.
By the definition \eqref{eqn:cG} of the Casimir element, one has
\[
    A^{(2)}
   =
   -\sum_{0<j\le n} X_{0j}^2
   =
   \sum_{j=1}^{n} X_{j0} X_{0j}
   =
   c_G - c_{G'}.
\]
Hence,
\[
   b^{(2)}(\lambda,\nu)
   =
   \chi_{\lambda,\nu}\!\left(c_G - c_{G'}\right)
=
   (\lVert \lambda\rVert^2-\lVert \rho_{\mathfrak g}\rVert^2)
   -
   (\lVert \nu\rVert^2-\lVert \rho_{G'}\rVert^2).
\]
Therefore,
\[
   b^{(2)}(\lambda,\nu)
   =
   \lVert \lambda\rVert^2
   -
   \lVert \nu\rVert^2
   -
   \frac{1}{8}n(n-1).
\]

\smallskip
\noindent{\rm(3)}\enspace
$\ell=3$.
A direct computation using \eqref{eqn:XX} yields
\begin{equation}
\label{eqn:25051117}
   A^{(3)}=(1-n)c_G + n c_{G'}.
\end{equation}
Consequently,
\[
   b^{(3)}(\lambda,\nu)
   =
   (1-n)\lVert \lambda\rVert^2
   +
   n \lVert \nu\rVert^2
   +
   \frac{1}{24}(n-1)n(2n-1).
\]
\end{example}

Recall from \eqref{eqn:WG} that
$W_G = \mathfrak S_r \ltimes (\mathbb{Z}/2\mathbb{Z})^r,
\
W_{G'} = \mathfrak S_s \ltimes (\mathbb{Z}/2\mathbb{Z})^s,
$
where $r=\lfloor \frac{n+1}{2}\rfloor$ and
$s=\lfloor \frac{n}{2}\rfloor$ are the ranks of $G$ and $G'$, respectively.

\begin{proposition}
\label{prop:241201}
Let $b^{(\ell)}(\lambda,\nu)$ be the polynomial defined in
\eqref{eqn:bN_lmdnu}, in the variables
\[
   \lambda=(\lambda_1,\dots,\lambda_r),
   \qquad
   \nu=(\nu_1,\dots,\nu_s).
\]
Then $b^{(\ell)}(\lambda,\nu)$ has degree at most $\ell$.
Moreover, $b^{(\ell)}(\lambda,\nu)$ is invariant under the action of the
product Weyl group $W_G \times W_{G'}$.
\end{proposition}
\begin{proof}
By the recursive definition \eqref{eqn:ANdef},
one has
$A^{(\ell)} \in U_{\ell}(\mathfrak{g})$.
Lemma~\ref{lem:25050506} implies that
$\chi_{\lambda,\nu}(A^{(\ell)})$ is a polynomial in $\lambda$ and $\nu$ of
degree at most $\ell$.

Furthermore, Proposition~\ref{prop:24120139} shows that
$A^{(\ell)} \in U(\mathfrak{g})^{G'}$.
Hence,
\[
   b^{(\ell)}(\lambda,\nu)
   =
   \chi_{\lambda,\nu}\bigl(A^{(\ell)}\bigr)
\]
is invariant under the action of $W_G \times W_{G'}$.
\end{proof}

\medskip
\subsection{Proof of Theorem~\ref{thm:25061808}}
\label{sec:pf_main}
~~~\newline
Let $\varepsilon \in \{-1,1\}$ and $1 \le i \le r$.
We recall the definition \eqref{eqn:phi_ae} of
\[
   \phi_{\lambda}^{\lambda+\varepsilon e_i}
   \in \mathfrak{Z}(G)[\lambda],
\]
which is a polynomial in the Casimir element $c_G$
with coefficients depending polynomially on $\lambda$.
\begin{lemma}
\label{lem:25051803}
Let $N \in \mathbb{N}_{+}$, $1 \le i \le r$, and
$\varepsilon \in \{+,-\}$.
Then there exists a unique polynomial
$p_{i,\varepsilon,N}(\lambda,\nu)$
in the variables $\lambda$ and $\nu$ with the following property:

For any real form
\[
   G \supset G'
   \quad\text{of}\quad
   O(n+1,\mathbb{C}) \supset O(n,\mathbb{C}),
\]
any representation
$\Pi \in \mathcal{M}(G)$ with
$\mathfrak{Z}(G)$-infinitesimal character $\lambda$,
any representation
$\pi \in \mathcal{M}(G')$ with
$\mathfrak{Z}(G')$-infinitesimal character $\nu$,
and any symmetry breaking operator
$T \colon \Pi \to \pi$,
the following identity holds:
\begin{equation}
\label{eqn:T_p_ieN}
   (T \otimes \operatorname{pr}_{F \to F'})
   \circ
   \bigl(\phi_{\lambda}^{\lambda+\varepsilon e_i}\bigr)^N
   \circ (\operatorname{id}_\Pi \otimes \iota_{F'\to F})
   =
   p_{i,\varepsilon,N}(\lambda,\nu)\,
   T.
\end{equation}
\end{lemma}

\begin{proof}
Using the shifted Casimir element $\widetilde c_G$
(see \eqref{eqn:cGtilde}), one has
\begin{equation}
\label{eqn:phi_ae2}
   \phi_{\lambda}^{\lambda+\varepsilon e_i}
   =
   \begin{cases}
   \displaystyle
   \prod_{(b,j) \in \Xi \setminus \{(\varepsilon,i)\}}
   \bigl(\widetilde c_G - b \lambda_j + \tfrac{n-1}{2}\bigr),
   & \text{if $n$ is odd}, \\[6pt]
   \displaystyle
   \bigl(\widetilde c_G + \tfrac{n}{2}\bigr)
   \prod_{(b,j) \in \Xi \setminus \{(\varepsilon,i)\}}
   \bigl(\widetilde c_G - b \lambda_j + \tfrac{n-1}{2}\bigr),
   & \text{if $n$ is even}.
   \end{cases}
\end{equation}
Here
\[
   \Xi = \{-1,1\} \times \{1,2,\dots,r\}.
\]

The assertion is therefore an immediate consequence of
Proposition~\ref{prop:TprCA}.
\end{proof}

We set
\begin{equation}
\label{eqn:P_ie}
   p_{i,\varepsilon}(\lambda,\nu)
   :=
   p_{i,\varepsilon,1}(\lambda,\nu).
\end{equation}

\begin{lemma}
\label{lem:25061804}
Let $p_{i,\varepsilon,N}(\lambda,\nu)$ be as defined in
Lemma~\ref{lem:25051803}.
Then, for any $N \in \mathbb{N}_{+}$, one has
\begin{equation}
\label{eqn:p_ieN}
   p_{i,\varepsilon,N}(\lambda,\nu)
   =
   \varphi_{i,\varepsilon}(\lambda)^{\,N-1}
   \, p_{i,\varepsilon}(\lambda,\nu),
\end{equation}
where $\varphi_{i,\varepsilon}(\lambda)$ is the polynomial defined in
\eqref{eqn:phi_ei_lmd}.
\end{lemma}
\begin{proof}
We prove the identity \eqref{eqn:p_ieN} on a Zariski-dense subset of the
$(\lambda,\nu)$-parameter space.

Recall from \eqref{eqn:piF_primary} that
\[
   \Pi \otimes F
   \simeq
   \bigoplus_{\tau \in \Delta(F)}
   P_{\lambda+\tau}(\Pi \otimes F)
\]
is the $\mathfrak{Z}(G)$-primary decomposition.
In particular, each primary component $P_{\lambda+\tau}(\Pi \otimes F)$ is a generalized
eigenspace of the Casimir element~$c_G$.

\medskip
\noindent
{\bf Step~1.\enspace The case of genuine eigenspaces.}
\par\noindent
Assume that the Casimir element $c_G$ acts by scalar multiplication on
each primary component
$P_{\lambda+\tau}(\Pi \otimes F)$ for every weight
$\tau \in \Delta(F)$.
It follows from Proposition~\ref{prop:25042720} that
\[
   \phi_{\lambda}^{\lambda+\varepsilon e_i}
   \bigl(\Pi \otimes F\bigr)
   \subset
   P_{\lambda+\varepsilon e_i}(\Pi \otimes F),
\]
and on this component
$\bigl(\phi_{\lambda}^{\lambda+\varepsilon e_i}\bigr)^{N-1}$
acts by the scalar
$\varphi_{i,\varepsilon}(\lambda)^{N-1}$.
Hence, by Lemma~\ref{lem:25051803},
\begin{align*}
   & (T \otimes \operatorname{pr}_{F \to F'})
   \circ
   \bigl(\phi_{\lambda}^{\lambda+\varepsilon e_i}\bigr)^N
   \circ
   (\operatorname{id}_\Pi \otimes \iota_{F'\to F})
\\
   &=
   \varphi_{i,\varepsilon}(\lambda)^{N-1}
   (T \otimes \operatorname{pr}_{F \to F'}) 
   \circ
   \phi_{\lambda}^{\lambda+\varepsilon e_i}
   \circ
   (\operatorname{id}_\Pi \otimes \iota_{F'\to F})
\\
   &=
   \varphi_{i,\varepsilon}(\lambda)^{N-1}
   p_{i,\varepsilon,1}(\lambda,\nu)\,
   T.
\end{align*}

\medskip
\noindent
{\bf Step~2.\enspace Reduction to the compact case.}
\par\noindent
The polynomial $p_{i,\varepsilon,N}(\lambda,\nu)$ depends only on the
infinitesimal characters and is independent of the choice of real forms
$(G,G')$ or representations $(\Pi,\pi)$.

We may therefore assume that
\[
   (G,G')=(O(n+1),O(n)),
\]
and take $\Pi$ to be an irreducible finite-dimensional representation of
$G$.

Since finite-dimensional representations of a compact group are
completely reducible, the Casimir element $c_G$ acts by scalars on each
primary component
$P_{\lambda+\tau}(\Pi \otimes F)$.
Consequently, the identity \eqref{eqn:p_ieN} holds 
whenever
\begin{itemize}
\item[$\bullet$]
$\lambda$ is the $\mathfrak{Z}(G)$-infinitesimal character of an
irreducible finite-dimensional representation $\Pi$ of $G=O(n{+}1)$,
\item[$\bullet$]
$\nu$ is the $\mathfrak{Z}(G')$-infinitesimal character of an
irreducible finite-dimensional representation of $G'=O(n)$ occurring in
the restriction $\Pi|_{G'}$.
\end{itemize}

Since the set of such pairs $(\lambda,\nu)$ is Zariski-dense in the
parameter space, the polynomial identity \eqref{eqn:p_ieN} holds
identically.
\end{proof}
In Section~\ref{sec:gi_compute}, we shall determine the explicit form of
the polynomial $p_{i,\varepsilon}(\lambda,\nu)$.
We state the result here.

\begin{proposition}
\label{prop:g_formula}
\[
   p_{i,\varepsilon}(\lambda,\nu)
   =
   g_{i,\varepsilon}(\lambda,\nu).
\]
\end{proposition}

Here $g_{i,\varepsilon}(\lambda,\nu)$ is given explicitly by
\eqref{eqn:gi_e}.

Postponing the proof of Proposition~\ref{prop:g_formula} to the next
section, we now complete the proof of
Theorem~\ref{thm:25061808}, and hence that of
all the main results, including
Theorem~\ref{thm:25061911a}.

\begin{proof}
[Proof of Theorem~\ref{thm:25061808}]
The theorem follows immediately from
Lemma~\ref{lem:25051803},
Lemma~\ref{lem:25061804}, and
Proposition~\ref{prop:g_formula}.
\end{proof}

\section{Computation of $p_{i,\pm}(\lambda,\nu)$}
\label{sec:gi_compute}

Determining the precise location of the zeros of the polynomial
$p_{i,\varepsilon}(\lambda,\nu)$ is a central issue in the analysis of
branching laws and symmetry breaking operators.
In this section, we prove Proposition~\ref{prop:g_formula}, namely, that
the polynomials $p_{i,\varepsilon}(\lambda,\nu)$ introduced in
Lemma~\ref{lem:25051803} are given explicitly by the polynomial
$g_{i,\varepsilon}(\lambda,\nu)$ defined in \eqref{eqn:gi_e}.
This computation thereby completes the proof of
Theorem~\ref{thm:25061808}, and hence the chain of our main theorems:
\[
   \text{Theorems~\ref{thm:25061911a} and~\ref{thm:25061911b}}
   \;\Longrightarrow\;
   \text{Theorem~\ref{thm:25051110}}
   \;\Longrightarrow\;
   \text{Theorem~\ref{thm:250515}}.
\]

Before proceeding further, we prepare a sequence of auxiliary lemmas.

We begin with some basic properties of the polynomials
$p_{i,\pm}(\lambda,\nu)$.

Recall that
\[
   W_G = \mathfrak S_r \ltimes (\mathbb{Z}/2\mathbb{Z})^r,
   \qquad
   W_{G'} = \mathfrak S_s \ltimes (\mathbb{Z}/2\mathbb{Z})^s,
\]
where $r=\left\lfloor \frac{n+1}{2}\right\rfloor$ and
$s=\left\lfloor \frac{n}{2}\right\rfloor$ are the ranks of $G$ and $G'$, respectively.
For each $1 \le i \le r$, let $W_i \subset W_G$ denote the stabilizer of $e_i$.
Let $s_i \in W_G$ be the reflection changing the sign of the $i$-th coordinate, that is,
\[
   s_i(\lambda_1,\dots,\lambda_r)
   =
   (\lambda_1,\dots,-\lambda_i,\dots,\lambda_r).
\]

\begin{lemma}
[Basic properties of $p_{i,\varepsilon}(\lambda,\nu)$]
\label{lem:25051805}
\rm
\begin{enumerate}
\item
The polynomials 
\newline
$p_{i,\pm}(\lambda,\nu)$ appearing in
Lemma~\ref{lem:25051803} have degree at most $n$.
\item
Viewed as polynomials in $\nu$,
$p_{i,\pm}(\lambda,\nu)$ are invariant under the Weyl group $W_{G'}$.
\item
Viewed as polynomials in $\lambda$,
$p_{i,\pm}(\lambda,\nu)$ are invariant under the subgroup
$W_i \subset W_G$.
\item
For every $1 \le i \le r$, one has
\[
   p_{i,+}(\lambda,\nu)
   =
   p_{i,-}(s_i\lambda,\nu).
\]
\end{enumerate}
\end{lemma}

\begin{proof}
\noindent
{\rm (1): \enspace (Degree bound).}\enspace
By the definition \eqref{eqn:phi_ae} of
$\phi_{\lambda}^{\lambda+\varepsilon e_i}$, we may write it as a linear
combination of terms of the form
\[
   q_N(\lambda)\,\widetilde c_G^{\,N},
   \qquad
   0 \le N \le n-1,
\]
where $q_N(\lambda)$ is a polynomial in
$\lambda_1,\dots,\lambda_r$ of degree at most $n-N$.
It then follows from Proposition~\ref{prop:TprCA} and
Proposition~\ref{prop:241201} that the polynomials
$p_{i,\pm}(\lambda,\nu)$ are of degree at most $n$.

\medskip
\noindent
{\rm (2) and (3): \enspace (Weyl group invariance).}\enspace
By Proposition~\ref{prop:241201}, the polynomial
$b^{(N)}(\lambda,\nu)$ is invariant under the action of
$W_G \times W_{G'}$.
Since $\phi_{\lambda}^{\lambda+\varepsilon e_i}\in\mathfrak{Z}(G)[\lambda]$
is invariant under the action of $W_i$ by definition
\eqref{eqn:phi_ae}, we conclude from the construction in
Lemma~\ref{lem:25051803} that $p_{i,\varepsilon,N}(\lambda,\nu)$ is
invariant under $W_i \times W_{G'}$.

{\rm (4):}
Applying the reflection $s_i$ to the defining equation
\eqref{eqn:phi_ae} of $\phi_{\lambda}^{\lambda+\varepsilon e_i}$, we
obtain
\[
   s_i\!\left(\phi_{\lambda}^{\lambda+\varepsilon e_i}\right)
   =
   \phi_{\lambda}^{\lambda-\varepsilon e_i}.
\]
It now follows from Lemma~\ref{lem:25051803} that
\[
   p_{i,-\varepsilon,N}(\lambda,\nu)
   =
   p_{i,\varepsilon,N}(s_i\lambda,\nu).
\]
The assertion~(4) follows by taking $\varepsilon=-$ and $N=1$.
\end{proof}
We therefore restrict our attention to the case
$(i,\varepsilon)=(1,-)$.
The structure of $p_{1,-}(\lambda,\nu)$ depends on the parity of $n$.

\begin{lemma}
\label{lem:25051806}
Let $s = \left\lfloor \frac{n}{2} \right\rfloor$ as before.
There exist a constant $c_{-}$ and a polynomial
$c_{-}(\lambda,\nu)$ satisfying
\[
   p_{1,-}(\lambda,\nu)
   =
   \begin{cases}
      \displaystyle
      c_{-}
      \prod_{j=1}^{s}
      \bigl(\lambda_1-\nu_j-\tfrac{1}{2}\bigr)
      \bigl(\lambda_1+\nu_j-\tfrac{1}{2}\bigr),
      & \text{if $n$ is even}, \\[1.2em]
      \displaystyle
      c_{-}(\lambda,\nu)
      \prod_{j=1}^{s}
      \bigl(\lambda_1-\nu_j-\tfrac{1}{2}\bigr)
      \bigl(\lambda_1+\nu_j-\tfrac{1}{2}\bigr),
      & \text{if $n$ is odd},
   \end{cases}
\]
where $c_{-}(\lambda,\nu)$ is a polynomial of degree at most one.
\end{lemma}

The polynomial $p_{1,-}(\lambda,\nu)$ in the variables $\lambda$ and $\nu$
satisfies \eqref{eqn:T_p_ieN} with $i=1$ for any pair of real forms $(G,G')$
of
\[
   (G_{\mathbb{C}},G_{\mathbb{C}}')
   =
   (O(n+1,\mathbb{C}),\,O(n,\mathbb{C}))
\]
and for any pair of representations $(\Pi,\pi)$ of these groups.

We therefore specialize to the case where $(G,G')$ is the pair of compact groups
$(O(n+1),O(n))$ for the proof of Lemma~\ref{lem:25051806}.
For the description of irreducible finite-dimensional
representations $\Pi$ and $\pi$ of $G$ and $G'$,
we use the standard Weyl notation
\[
   \Pi = F^{O(n+1)}(\mu),
   \qquad
   \pi = F^{O(n)}(\mu'),
\]
labeled with
$\mu=(\mu_1,\dots,\mu_{n+1})$ and
$\mu'=(\mu_1',\dots,\mu_n')$.

For instance, $F^{O(n+1)}(e_1)$ is the standard representation $\mathbb{C}^{n+1}$.

\begin{proof}[Proof of Lemma~\ref{lem:25051806}] 
Let $\Pi = F^{O(n+1)}(\mu)$ and  $\pi = F^{O(n)}(\mu')$ be as above. Then the infinitesimal characters of $\Pi$ and $\pi$
are given by $\lambda=(\lambda_1,\dots,\lambda_r)$
and $\nu=(\nu_1,\dots,\nu_s)$, where
\begin{align*}
   \lambda_i
   &= \mu_i + \frac{n+1}{2} - i,
   &\qquad& \text{for $1 \le i \le r$,}
\\
   \nu_j
   &= \mu_j' + \frac{n}{2} - j,
   &\qquad& \text{for $1 \le j \le s$.}
\end{align*}

In particular,
\[
   \lambda_i - \nu_i
   =
   \mu_i - \mu_i' + \frac{1}{2},
   \qquad
   \text{for $1 \le i \le s$.}
\]
Suppose that $\mu_1 > \mu_2$.
Then
\[
P_{\lambda-e_1}
(\Pi \otimes \mathbb{C}^{n+1})
\simeq
F^{O(n+1)}(\mu - e_1).
\]
Moreover, if $\mu_1 = \mu_1'$, then
\[
   \operatorname{Hom}_{G'}
   \bigl(
      P_{\lambda-e_1}(\Pi \otimes \mathbb{C}^{n+1})|_{G'},\,
      \pi
   \bigr)
   \simeq
   \operatorname{Hom}_{G'}
   \bigl(
      F^{O(n+1)}(\mu - e_1)|_{G'},\,
      \pi
   \bigr)
   = \{0\}.
\]
Hence the polynomial $p_{1,-}(\lambda,\nu)$ vanishes whenever
$\mu_1 = \mu_1'$, that is, whenever
\[
   \lambda_1 - \nu_1 - \frac{1}{2} = 0.
\]

Since $p_{1,-}(\lambda,\nu)$ is invariant under the Weyl group
$W_{G'}$, it follows that $p_{1,-}(\lambda,\nu)$ is divisible by
\[
   \prod_{j=1}^{s}
   \bigl(\lambda_1 - \nu_j - \tfrac{1}{2}\bigr)
   \bigl(\lambda_1 + \nu_j - \tfrac{1}{2}\bigr).
\]

Finally, Proposition~\ref{prop:241201} shows that
$p_{1,-}(\lambda,\nu)$ is a polynomial in $\lambda$ and $\nu$ of degree at
most $n$.
This completes the proof of Lemma~\ref{lem:25051806}.
\end{proof}
The following lemma provides further information on the constant
$c_{-}$ and the linear factor $c_{-}(\lambda,\nu)$ appearing in
Lemma~\ref{lem:25051806}, which treated only the case $i=1$ and
$\varepsilon=-$, by clarifying their dependence on
$1 \le i \le r$ and $\varepsilon \in \{+,-\}$.

\begin{lemma}
\label{lem:250518}
There exist constants $c_{+}$, $p$, and $q$, independent of
$1 \le i \le r$ and $\varepsilon \in \{+,-\}$, such that
\[
   p_{i,\varepsilon}(\lambda,\nu)
   =
   \begin{cases}
      \displaystyle
      c_{+}
      \prod_{j=1}^{s}
      \bigl(\lambda_i-\nu_j+\tfrac{1}{2}\varepsilon\bigr)
      \bigl(\lambda_i+\nu_j+\tfrac{1}{2}\varepsilon\bigr),
      & \text{if $n$ is even}, \\[1.2em]
      \displaystyle
      (\varepsilon p\,\lambda_i + q)
      \prod_{j=1}^{s}
      \bigl(\lambda_i-\nu_j+\tfrac{1}{2}\varepsilon\bigr)
      \bigl(\lambda_i+\nu_j+\tfrac{1}{2}\varepsilon\bigr),
      & \text{if $n$ is odd}.
   \end{cases}
\]
\end{lemma}

\begin{proof}
By Lemma~\ref{lem:25051805}, it suffices to establish the formula for
$p_{1,-}(\lambda,\nu)$.

If $n$ is even, the assertion follows directly from
Lemma~\ref{lem:25051806}.

Assume now that $n=2r-1$ is odd.
The product
\[
   \prod_{j=1}^{s}
   \bigl(\lambda_1-\nu_j-\tfrac{1}{2}\bigr)
   \bigl(\lambda_1+\nu_j-\tfrac{1}{2}\bigr)
\]
is invariant under $W_1 \times W_{G'}$.
Hence the polynomial $c_{-}(\lambda,\nu)$ is also invariant under
$W_1 \times W_{G'}$.
Since $c_{-}(\lambda,\nu)$ has degree at most one, it cannot depend on
$\lambda_2,\dots,\lambda_r$ or on $\nu_1,\dots,\nu_{r-1}$.
Therefore $c_{-}(\lambda,\nu)$ must be linear in $\lambda_1$, and is of
the form
\[
   -p\,\lambda_1 + q
\]
for some constants $p$ and $q$.
This completes the proof.
\end{proof}

The proof of Proposition~\ref{prop:g_formula} will be complete once we
determine the constants $p$, $q$, and $c_{+}$ appearing in
Lemma~\ref{lem:250518}.

\begin{lemma}
\label{lem:25051002}
Let $n$ be even.
Then one has $c_{+}=1$.
\end{lemma}

\begin{proof}
Let $n=2r$.
Since $G$ is compact, each primary component of $\Pi \otimes F$ is a
genuine eigenspace of $\mathfrak{Z}(G)$.
Hence, by Proposition~\ref{prop:25042720}, one has
\[
   \phi_{\lambda}^{\lambda+e_1}
   =
   \lambda_1(2\lambda_1+1)
   \prod_{j=2}^{r}
   (\lambda_1-\lambda_j)(\lambda_1+\lambda_j)
   \, P_{\lambda+e_1}.
\]
It then follows from Lemma~\ref{lem:25051803} with $N=1$
and Lemma~\ref{lem:250518} that
\begin{align}
\notag
   &\lambda_1(2\lambda_1+1)
   \prod_{j=2}^{r}
   (\lambda_1-\lambda_j)(\lambda_1+\lambda_j)\,
   (T \otimes \operatorname{pr}_{F \to F'})
   \circ
   P_{\lambda+e_1}(u \otimes f_0)
\\
\label{eqn:25051002}
   &=
   c_{+}
   \prod_{j=2}^{r}
   \Bigl(\lambda_1-\nu_j+\tfrac{1}{2}\Bigr)
   \Bigl(\lambda_1+\nu_j+\tfrac{1}{2}\Bigr)
   \, T(u) \otimes f_0 .
\end{align}

To determine the constant $c_{+}$, we take
$\Pi=\mathbf{1}$ and $\pi=\mathbf{1}$ to be the trivial
one-dimensional representations of $G$ and $G'$, respectively.
In this case, the $\mathfrak{Z}(G)$- and $\mathfrak{Z}(G')$-infinitesimal
characters are given by
\[
   \lambda
   =
   \bigl(r-\tfrac{1}{2},\dots,\tfrac{3}{2},\tfrac{1}{2}\bigr),
   \qquad
   \nu
   =
   (r-1,\dots,1,0),
\]
respectively.
Take
\[
   \sigma := F^{O(n+1)}(e_1)\simeq F=\mathbb{C}^{n+1},
\]
so that the projection
$P_{\lambda+e_1} \colon \Pi \otimes F \to \sigma$
is the identity map.

The left-hand side of \eqref{eqn:25051002} evaluates to
\[
   r(2r-1)!\,
   (T \otimes \operatorname{pr}_{F \to F'})
   \circ
   P_{\lambda+e_1}(u \otimes f_0)
   =
   r(2r-1)!\,T(u) \otimes f_0,
\]
whereas the right-hand side equals
\[
   c_{+}\, r(2r-1)!\,T(u) \otimes f_0.
\]
Comparing the two expressions, we conclude that $c_{+}=1$.
\end{proof}

\begin{lemma}
\label{lem:25051016}
Let $n$ be odd.
Then one has $p=1$ and $q=0$.
\end{lemma}

\begin{proof}
Let $n=2r-1$. Note that $s=\lfloor \frac{n}2 \rfloor =r-1$.
As in the proof of Lemma~\ref{lem:25051002}, since $G$ is compact,
Proposition~\ref{prop:25042720} yields
\[
   \phi_{\lambda}^{\lambda+e_1}
   =
   2\lambda_1
   \prod_{j=2}^{r}
   (\lambda_1-\lambda_j)(\lambda_1+\lambda_j)
   \operatorname{pr}_{\sigma},
\]
where $\sigma=F^{O(n+1)}(e_1)$.

Again by Lemma~\ref{lem:25051803} with $N=1$
and Lemma~\ref{lem:250518} (with $\varepsilon=+$), we obtain
\begin{multline}
\label{eqn:25051005}
   2\lambda_1
   \prod_{j=2}^{r}
   (\lambda_1-\lambda_j)(\lambda_1+\lambda_j)\,
   (T \otimes \operatorname{pr}_{F \to F'})
   \circ
   P_{\lambda+e_1}(u \otimes f_0)
  \\ =
   (p\lambda_1+q)
   \prod_{j=1}^{r-1}
   \Bigl(\lambda_1-\nu_j+\tfrac{1}{2}\Bigr)
   \Bigl(\lambda_1+\nu_j+\tfrac{1}{2}\Bigr)
   \, T(u) \otimes f_0.
\end{multline}

We examine \eqref{eqn:25051005} in two cases.

\medskip
\noindent
{\bf Case~1.}\enspace
Take $\Pi=\mathbf{1}$, $\pi=\mathbf{1}$, and
$\sigma=F^{O(n+1)}(e_1) \simeq \mathbb{C}^{n+1}$.

In this case, the $\mathfrak{Z}(G)$- and $\mathfrak{Z}(G')$-infinitesimal
characters of $\Pi$ and $\pi$ are given by
\[
   \lambda=(r-1,r-2,\dots,1,0),
   \qquad
   \nu=\bigl(r-\tfrac{3}{2},\dots,\tfrac{3}{2},\tfrac{1}{2}\bigr).
\]
Substituting these into \eqref{eqn:25051005}, we obtain
\[
   (r-1)(2r-2)!
   =
   \bigl((r-1)p+q\bigr)(2r-2)!,
\]
and hence
\begin{equation}
\label{eqn:25051016a}
   r-1=(r-1)p+q.
\end{equation}
\par\noindent
\medskip
\noindent
{\bf Case~2.}\enspace
Take $\Pi=F$, $\pi=\mathbf{1}$, and
$\sigma=F^{O(n+1)}(2 e_1) \simeq S^2(F)/\mathbf{1}$.

In this case, the $\mathfrak{Z}(G)$- and $\mathfrak{Z}(G')$-infinitesimal
characters of $\Pi$ and $\pi$ are given by
\begin{equation}
\label{eqn:lmd_2}
   \lambda=(r,r-2,\dots,1,0),
   \qquad
   \nu=\bigl(r-\tfrac{3}{2},\dots,\tfrac{3}{2},\tfrac{1}{2}\bigr).
\end{equation}
The tensor product representation
$\mathbb{C}^{n+1} \otimes \mathbb{C}^{n+1}$ decomposes as
\begin{align*}
   \mathbb{C}^{n+1} \otimes \mathbb{C}^{n+1}
   &=
   S^2(\mathbb{C}^{n+1}) \oplus \Lambda^2(\mathbb{C}^{n+1})
\\
   &\simeq
   \bigl(
      F^{O(n+1)}(2e_1)
      \oplus
      F^{O(n+1)}(0)
   \bigr)
   \oplus
   F^{O(n+1)}(e_1+e_2).
\end{align*}
According to this decomposition, one has
\[
   f_0 \otimes f_0
   =
   \frac{1}{n+1}
   \Bigl(
      n\,f_0 \otimes f_0
      -
      \sum_{j=1}^{n} f_j \otimes f_j
   \Bigr)
   +
   \frac{1}{n+1}
   \sum_{j=0}^{n} f_j \otimes f_j,
\]
and hence
\[
   P_{\lambda+e_1}(f_0 \otimes f_0)
   =
   \frac{1}{n+1}
   \Bigl(
      n\,f_0 \otimes f_0
      -
      \sum_{j=1}^{n} f_j \otimes f_j
   \Bigr).
\]
Therefore,
\[
   (T \otimes \operatorname{pr}_{F \to F'})
   \circ
   P_{\lambda+e_1}(f_0 \otimes f_0)
   =
   \frac{1}{n+1}\, T(f_0) \otimes f_0.
\]

Substituting \eqref{eqn:lmd_2} into \eqref{eqn:25051005}, we obtain
\[
   2r^2(2r-2)!\,\frac{n}{n+1}\,T(f_0) \otimes f_0
   =
   (rp+q)(2r-1)!\,T(f_0) \otimes f_0.
\]
Since $n=2r-1$, this yields
\begin{equation}
\label{eqn:25051016b}
   r = rp + q.
\end{equation}
Solving the system of linear equations
\eqref{eqn:25051016a} and \eqref{eqn:25051016b}, we conclude that
$p=1$ and $q=0$.
\end{proof}
This completes the proofs of the main theorems.
The argument shows that, although translation functors for higher-rank
orthogonal groups are in general governed by the full structure of the
center $\mathfrak{Z}(G)$, their effect on symmetry breaking operators can
in fact be controlled by explicit polynomials in the Casimir element.
This reduction yields a uniform scalar description of symmetry breaking
and explains the nonvanishing and factorization phenomena established in
Theorems~\ref{thm:25061911a} and~\ref{thm:25061911b} through
Theorem~\ref{thm:25061808}, where these properties are formulated in terms
of polynomials in the Casimir element.
As a consequence, the Stability of Multiplicity—formulated via fences in
the parameter space of \emph{reduced coherent families} and stated in
Theorem~\ref{thm:250515}—follows directly from these scalar identities,
providing a conceptual explanation of multiplicity stability
within each convex region of the parameter space determined by the fences.

\vskip 1pc
\par\noindent
{\bf{Acknowledgement.}}
\newline
The author is grateful to the Institut Henri Poincaré (IHP) in Paris and the Institut des Hautes Études Scientifiques (IHES) in Bures-sur-Yvette, where part of this work was carried out.
He was partially supported by the JSPS
 under the Grant-in Aid for Scientific Research (A) 
 (JP23H00084).  

\medskip
\section*{Statements and Declarations}

\textbf{Competing Interests:}
The author declares that there are no competing interests.

\end{document}